 \documentclass[11pt]{article}

\usepackage{amsfonts,amsmath,amssymb,amsthm}
\usepackage{epsfig}
\usepackage{graphics}
\usepackage{graphicx}
\usepackage{times}
\usepackage{mathrsfs}

\title{Zero-density estimates for Epstein zeta functions \footnote{Research of the first author was partially supported by NSF grant DMS-1200582. Research of the second author was supported by the Incheon National University Research Grant in 2014.}}

\begin{document}
\newcommand\tabcaption{\def\@captype{table}\caption}
\newtheorem{thrm}{Theorem}
\newtheorem*{thmA}{Theorem A}
\newtheorem*{thmB}{Theorem B}
\newtheorem{cor}[thrm]{Corollary}
\newtheorem{lemma}{Lemma}
\newtheorem{prop}{Proposition}
\newtheorem{conj}{Conjecture}
\newtheorem{defn}{Definition}
\newtheorem*{ack}{Acknowledgement}
\newtheorem*{key}{Keywords}
\newtheorem*{msc}{Mathematics Subject Classification (2010)}
\newcommand\be{\begin{equation}}
\newcommand\ee{\end{equation}}

\numberwithin{equation}{section}

\def\k{{k_0}}
\def\lin{Lindel\"{o}f }
 \def\s{\text{\rm sign}}
\def\u{\mathbf u}
\def\v{\mathbf v}
 \def\A{\mathcal A}
 \def\B{\mathscr B}
 \def\D{\mathbf D}
  \def\R{\mathscr R}
 \renewcommand{\S}{\mathcal S}
 \newcommand{\I}{\mathcal I}
  \newcommand{\J}{\mathcal J}
  \def\sL{\mathscr L}
\def\E{\mathbb E}
\def\L{\mathbf L}
\def\P{\mathbb P}
\def\m{\mathbf m}
\def\n{\mathbf n}
\def\w{\mathbf w}
\def\z{\mathbf z}
\def\a{\mathbf a}
\def\s{\sigma}
\def\bysame{\leavevmode\hbox to3em{\hrulefill}\thinspace}

\author{Steven Gonek  and Yoonbok Lee }

\parskip=12pt

 \maketitle 
 \pagestyle{myheadings}

 \begin{abstract}
  We investigate the zeros of Epstein zeta functions associated with a positive definite quadratic form with rational coefficients in the vertical strip $ \sigma_1 < \Re s <  \sigma_2 $, where $ 1/2 < \sigma_1 < \sigma_2 < 1 $. When the class number of the quadratic form is bigger than 1, Voronin gives a lower bound and Lee gives an asymptotic formula for the number of zeros. In this paper, we improve their results by providing a new upper bound for the error term.
 \end{abstract}

 

\begin{msc}
11M26, 11M41.
\end{msc}


\section{Introduction and statement of results}


Let $ Q(m,n) = am^2 + bmn + cn^2 $ be a positive definite quadratic form with $ a,b,c \in \mathbb{Z}$ and $ D= b^2 - 4ac <0$. Let $s=\sigma+it$ be a complex variable. The Epstein zeta function associated with $Q$ is defined by 
$$ E(s,Q) = {\sum_{m,n}}'     \frac{1}{ Q(m,n)^s }  $$
for $ \sigma > 1 $, where the sum is over all integers $m,n$ not both zero. It has a meromorphic continuation to $\mathbb{C}$ with a simple pole at $s=1$ and it satisfies the functional equation
\begin{equation}\label{func eqn}
\Psi(s , Q) = \Psi (1-s, Q),
\end{equation}
where
$$ \Psi(s,Q) : = \left( \frac{  \sqrt{-D}}{2 \pi} \right)^s \Gamma (s) E(s,Q)  . $$

In this paper we study the distribution of the zeros of $E(s,Q)$ in the right half of the critical strip, $1/2<\sigma<1$. This distribution is different depending on whether 
 the   class number $h(D)$ of $ \mathbb{Q}(\sqrt{D})$ is $1$ or is greater than $1$. 
 If $h(D) = 1 $, then  
$$ E(s,Q) = w_D \zeta_K (s) ,$$
where $ w_D$ is the number of roots of unity in $K = \mathbb{Q}(\sqrt{D})$ and $ \zeta_K$ is the Dedekind zeta function of $K$. Hence, in this case we  expect  $E(s,Q)$ to satisfy the analogue of the Riemann hypothesis. However, if $h(D) >1 $, Davenport and Heilbronn \cite{DH} proved that $E(s,Q)$ has infinitely many zeros on $\sigma > 1 $.
For $ 1/2 < \sigma_1 < \sigma_2  $ set
\be\notag
N_E(\s_1, \s_2; T) =\sum_{\substack{T\leq  \gamma \leq 2T \\ \s_1<\beta\leq \s_2}}1,
\ee
where $\rho=\beta+i\gamma$ denotes a generic zero of $E(s,Q)$.
Voronin \cite{V} proved the following theorem.
 
\begin{thmA}[Voronin]\label{V thrm}
Let $ Q$ be a quadratic form with integer coefficients whose discriminant is $D<0$. Suppose that the class number $h(D)$ is greater than $1$.  Then for any $ \sigma_1 $ and $ \sigma_2$ with $ 1/2 < \sigma_1 < \sigma_2 < 1 $, 
$$
N_E(\s_1, \s_2; T) \geq cT,
$$
where $ c = c( \sigma_1, \sigma_2 , Q)>0$ is independent of $T$.
\end{thmA}

Recently, the second author \cite{Le2} improved this to an asymptotic formula.
\begin{thmB}[Lee]\label{thm: Lee}
Assume the same hypothesis as in Theorem A. Then  for
$ 1/2 < \sigma_1 < \sigma_2 $ we have
$$
N_E(\s_1, \s_2; T) = cT+o(T), 
$$ 
where  $ c = c( \sigma_1, \sigma_2, Q) \geq 0 $. If $ \sigma_1 \leq 1 $, then $ c>0$.
\end{thmB}

Our main theorem provides an 
improvement of the error term. 
\begin{thrm}\label{main thrm}
Assume the same hypothesis as in Theorem A. 
If $ 1/2 < \sigma_1 < \sigma_2 $,
then there exists an absolute constant $b>0$ such that   
$$
N_E(\s_1, \s_2; T) =
cT+O(T\exp(- b \sqrt{\log \log T} )),
$$ 
 where $ c = c( \sigma_1, \sigma_2, Q) \geq 0 $. If $ \sigma_1 \leq 1 $, then $ c>0$.
\end{thrm}

The proofs of the above theorems begin with the well-known identity
\begin{equation}\label{E lin com eqn}
 E(s,Q) = \frac{w_D}{ h(D)} \sum_{\chi} \overline{\chi (\frak{a}_Q ) } L(s, \chi) , 
 \end{equation}
where the sum is over all characters of the class group, 
$\frak{a}_Q$ is any integer ideal in the ideal class corresponding to the equivalence class
of $Q$,  
 and $L(s,\chi)$ is the Hecke $L$-function defined by
$$ L(s,\chi) = \sum_{\frak{n}} \frac{ \chi( \frak{n}) }{ \frak{N}( \frak{n} )^s} = \prod_{\frak{p}} \left( 1- \frac{ \chi( \frak{p})}{\frak{N}( \frak{p})^s} \right)^{-1}   $$
for $ \s > 1 $. Here $ \frak{N}$ is the norm. Each Hecke $L$-function has a meromorphic continuation to $\mathbb{C}$, and it has a simple pole at $s=1$ only when the character $\chi$ is trivial. It also satisfies the   functional equation  \eqref{func eqn}
except that this time 
$$ 
\Psi(s,\chi):= \left( \frac{  \sqrt{-D}}{2 \pi} \right)^s \Gamma (s) L(s,\chi)  . 
$$
The $L$-functions in the sum \eqref{E lin com eqn} are not distinct.   For each rational prime $p$, a principal ideal $ (p)$ is a prime ideal $ \frak p$ or a product of two prime ideals  $ \frak p_1 \frak p_2 $. If $ (p) = \frak p $, then $ \chi( \frak p ) = 1 $ and 
$$  \prod_{\frak{p} | p } \left( 1- \frac{ \chi( \frak{p})}{\frak{N}( \frak{p})^s} \right)^{-1} =  \left( 1- \frac{1}{p^{2s}} \right)^{-1} .  $$
If $ (p) = \frak p_1 \frak p_2 $,  then $ \chi( \frak p_1 ) \chi( \frak p_2 ) = 1 $. Thus $ \chi ( \frak p_1 ) = \overline{ \chi ( \frak p_2 ) } $ and 
$$  \prod_{\frak{p} | p } \left( 1- \frac{ \chi( \frak{p})}{\frak{N}( \frak{p})^s} \right)^{-1} =  \left( 1- \frac{ \chi ( \frak p_1 ) + \chi( \frak p_2 )  }{p^s}   + \frac{1}{ p^{2s}}   \right)^{-1} = \left( 1- \frac{2 \Re \chi ( \frak p_1 )    }{p^s}   + \frac{1}{ p^{2s}}   \right)^{-1} .  $$
It follows that $ L(s,\chi ) = L ( s, \overline{\chi}) $.
We now let $ J $ be the number of real characters plus one-half the number of complex characters, and list the characters as $\chi_1 ,\dots, \chi_J $ in such a way that $ \chi_j \neq \chi_k $ and $ \chi_j \neq \overline \chi_k $ for $ j \neq k$. Then, writing
$$
L_j(s) =L(s, \chi_j),
$$
we may   rewrite   \eqref{E lin com eqn} as
\begin{equation}\label{lin com repn}
 E(s,Q) =  \sum_{j=1}^J  c_j  L_j (s) , 
 \end{equation}
 where  
 $$ c_j =  \frac{w_D}{ h(D)} \chi_j ( \frak{a}_Q ) $$
 for real characters $\chi_j$, and
 $$ c_j = \frac{w_D}{ h(D)} 2 \Re \chi_j ( \frak{a}_Q ) $$
 for complex characters $ \chi_j$. Note that $ J>1$ if and only if $ h(D) > 1 $.

 Voronin~\cite{V} deduced Theorem~\ref{V thrm} from a joint distribution result for  the inequivalent Hecke $L$-functions  $L_1 (s) , \dots, L_J (s)$ in \eqref{lin com repn}. 
 On the other hand, Lee's proof of Theorem~\ref{thm: Lee} in \cite{Le2} proceeded via a study of the Jensen function
$$ \varphi (\sigma ) = \lim_{T \to \infty} \frac 1T \int_1^T \log | E( \sigma + it ,Q) | dt .
$$ 
Lee    showed that $\varphi (\sigma )$  is twice differentiable and that the density of zeros of $E(s)$ in the strip $ \sigma_1 < \s < \sigma_2 $ equals   
$$ \varphi'(\sigma_2 ) - \varphi'(\sigma_1 ) = \int_{\sigma_1}^{\sigma_2 } \varphi''(\sigma) d \sigma . 
$$
Our proof of Theorem~\ref{main thrm}  also  proceeds through the estimation of the  integral 
$$\frac1T  \int_T^{2T} \log | E( \sigma + it , Q ) | dt = \frac1T  \int_T^{2T}
 \log \bigg| \sum_{j \leq J } c_j  L_j ( \sigma  + it)  \bigg|  dt  $$
 but, in addition,  incorporates   recent ideas of Lamzouri, Lester, and Radziwi\l\l~\cite{LLR} 
in their study of the distribution of $a$-points of the Riemann zeta function.

It is worth noting that when  $ J =2$, 
\begin{align*}
 \frac1T  \int_T^{2T} \log | E( \sigma + it , Q ) | dt 
 & = \frac1T \int_T^{2T} \log | c_1 L_1 ( \sigma+it)   | dt  + \frac1T \int_T^{2T} \log \left|  \frac{ c_2 L_2 ( \sigma+it)}{c_1 L_1 ( \sigma + it) } + 1 \right| dt  .
 \end{align*}
It is not difficult to show that the first term here equals  $\log |c_1 |$ plus a  small   error term. We can also estimate the second term by a   straightforward   adaptation of the method in \cite{LLR}. However,  when $ J > 2 $   this approach   no longer   works. In what follows, therefore,  we are mostly  interested in the case $ J \geq 3 $.

 Corresponding to the Hecke $L$-functions
$$ 
L_j(s)= L(s,\chi_j) = \prod_{\frak{p}} \left( 1- \frac{ \chi_j( \frak{p})}{\frak{N}( \frak{p})^s} \right)^{-1}   \qquad (1\leq j\leq J)
$$
we define  the random models  
$$ 
L_j(\sigma, X) :=   \prod_{\frak{p}} \left( 1- \frac{ \chi_j( \frak{p})X( {p})}{\frak{N}( \frak{p})^\sigma} \right)^{-1}  \qquad (1\leq j\leq J),  
$$
where $p$ is the unique rational prime dividing $\frak N(\frak p)$, and the $X(p)$ are uniformly and independently distributed on the unit circle $\mathbb T$. Note that these products converge almost surely on $\mathbb T^\infty$  for $\s> 1/2$.
We define 
\be\notag
\log L_j(\sigma, X):= \sum_{\frak p} \sum_{k=1}^{\infty}  \frac{\chi_j( \frak{p}^k)\  
X( {p})^k}{k\ \frak{N}( \frak{p})^{k\sigma} }
\ee
and define $\log |L_j(\sigma, X)|$ and $\arg L_j(\sigma, X)$ as its real and imaginary parts, 
respectively.
These too converge almost surely on $\mathbb T^\infty$ for $\s> 1/2$.

Let
 \begin{align*}
  \L(s)  :=  
  \big(\log & |{c_1 L_1}/{c_JL_J} (s)|,   \dots ,  
 \log  | {c_{J-1}L_{J-1}}/{c_J L_J} (s)  |, \log  | c_J L_J (s)  |   ;  \\
&  \arg  {c_1L_1}(s)-\arg{c_J L_J} ( s) , \dots ,  
  \arg  {c_{J-1} L_{J-1}}(s) - \arg{c_J L_J} (s), \arg c_J L_J ( s) \big) 
 \end{align*}
     and
     \begin{align*}
 \L  ( \sigma, X )  := \big(    \log &|{c_1 L_1}/{c_J L_J}  ( \sigma,X)| ,\dots,  
  \log  | c_{J-1} L_{J-1} /{c_JL_J}(\sigma , X)  | ,   \log  |  {c_JL_J} (\sigma , X) | ; \\
&     \arg   c_1 L_1( \sigma, X )   -\arg c_JL_J( \sigma, X ) , \dots , 
\arg  c_{J-1} L_{J-1}( \sigma, X ) -\arg c_J L_J( \sigma, X ) , \\
&\quad \arg c_J L_J ( \sigma , X)  \big) .  
\end{align*}
For  a Borel set $  \B$ in $\mathbb R^{2J}$ and for~$1/2<\sigma<1$ fixed, we define
\be\label{Psi_T}
\Psi_T(\B)
 := \frac1T \mathrm{meas}  \{ t\in[T, 2T] :  \L (s)\in \B  \}  
\ee
and 
 \be\label{Psi}
\Psi(\B):=
 \mathbb{P}(\L(\sigma, X) \in \B)
 =  \mathrm{meas}  \{ X  \in \mathbb{T}^\infty :  \L(\sigma, X) \in \B  \}  .
 \ee
  We define the \emph{discrepancy} between these two distributions as
$$
 \D_T(\B):= \Psi_T(\B)-\Psi(\B).
$$ 
The key ingredient in our proof of Theorem~\ref{main thrm} is the  following bound  for
$ \D_T(\B)$, which is an analogue of Theorem 1.1 of \cite{LLR}   and is interesting 
in its own right.

\begin{thrm}\label{lem: discrepancy} 
Let $\tfrac12< \sigma<1$ be fixed. Then
$$
 \sup_{\mathscr R} |\D_T(\mathscr R)| \ll \frac{1}{(\log T)^{\sigma}},
$$
where $\mathscr R$ runs over all rectangular regions (possibly unbounded) with sides parallel to the coordinate axes.
\end{thrm}

The letters   $A, B$ and $C$  denote positive constants throughout that are not necessarily  the same at each occurrence. Boldfaced letters denote vectors whose components may be functions.
We also write
$$
\sL =\log\log T.
$$

\section{Basic lemmas}
In this section we provide several  of the  technical lemmas we shall need later.
 
 \begin{lemma}\label{lem: log 2k dv}
Let $ a\geq b>0$. There exists an absolute positive constant $C$ such that for 
any positive integer $k$ we have
$$ \frac{1}{ 2 \pi} \int_0^{2 \pi} \big(   \log  |a - be^{iv}  | \big)^{2k} dv 
\ll   (C\  |\log a|)^{2k} + ( Ck)^{2k} . $$

\end{lemma}

\begin{proof}
First assume that $ a>b $. Then writing $z$ for $e^{iv}$, we find that 
\begin{align}\label{aux mmt}
 & \frac{1}{ 2 \pi} \int_0^{2 \pi} \big(   \log  | a- be^{iv}  | \big)^{2k} dv   \notag  \\
& =   \frac{1}{(2 \pi i)  4^k } \int_{|z|=1} \big(   - \log ( a- bz ) -     \log  ( a-  {b}{z^{-1}}  )  \big)^{2k} \frac{dz}{z} \\
& =  \frac{1}{(2 \pi i)  4^k }  \int_{|z|=1} \bigg( -2 \log a + \sum_{n=1}^\infty \frac{( bz/a)^n }{n}     
+  \sum_{m=1}^\infty \frac{  (  b/az)^m}{m}     \bigg)^{2k}\ \frac{dz}{z}   \notag  \\
& = \frac 1{4^k}  \sum_{k_1 + k_2 + k_3 = 2k} { {2k} \choose { k_1 , k_2, k_3 } } 
( -2 \log a )^{k_1}  \Bigg\{ \frac{1}{ 2 \pi i } \int_{|z|=1}  \bigg( \sum_{n=1}^\infty \frac{ (bz/a)^n}{n}   \bigg)^{k_2} 
\bigg(  \sum_{m=1}^\infty \frac{(b/az)^m}{m}   \bigg)^{k_3 } \frac{dz}{z} \Bigg\} .  \notag 
\end{align}
We calculate the expression in braces by the calculus of residues. If $k_2=k_3=0$, it equals $1$. If one of $k_2, k_3 $ 
is $0$ but the other is not,   it equals $0$. In all other cases, since $a>b>0$,
 it equals
\be\notag
\begin{split}
 \sum_{ n_1 + \cdots + n_{k_2} = m_1 + \cdots + m_{k_3}}   
& \frac{(b/a)^{n_1 + \cdots + n_{k_2} + m_1 + \cdots + m_{k_3} }}{  n_1  \cdots  n_{k_2}  m_1 \cdots  m_{k_3} } \\
  &\qquad <\sum_{ n_1 + \cdots + n_{k_2} = m_1 + \cdots + m_{k_3}}   
 \frac{1}{  n_1  \cdots  n_{k_2}  m_1 \cdots  m_{k_3} } \\
  &\qquad =\sum_{\ell = 1}^\infty  \Bigg( \sum_{ n_1 + \cdots + n_{k_2} =\ell} \frac{1}{  n_1  \cdots  n_{k_2} } \Bigg) 
  \Bigg(   \sum_{ m_1 + \cdots + m_{k_3}=\ell}   \frac{1}{  m_1 \cdots  m_{k_3} } \Bigg).
\end{split}
\ee
In the first sum at least one of the $n_i$ is  maximal, and    therefore $ \geq \ell/k_2 $. There are $k_2$
choices for the maximal term, so
\begin{align*}
 \sum_{ n_1 + \cdots + n_{k_2} =\ell} \frac{1}{  n_1  \cdots  n_{k_2} } 
  \leq k_2 \frac{1}{ (\ell/k_2) }  \sum_{n_1 , \dots,\ n_{k_2 - 1} \leq \ell}   \frac{1}{n_1  \cdots \  n_{k_2 -1}}  
  \ll     \frac{ k_2 ^2 }{ \ell}   ( \log \ell)^{k_2 -1}  .
\end{align*}
Thus the above is
\be\notag
(k_2 k_3)^2  \sum_{\ell = 1}^\infty   \frac{( \log \ell)^{k_2 +k_3-2}  }{ \ell^2}   .
\ee
Combining our estimates, we find that  \eqref{aux mmt} is
\be\notag
\ll \frac 1{4^k}(2|\log a|)^{2k}  +
 \frac {k^4  }{4^k}  \sum_{\substack{k_1 + k_2 + k_3 = 2k\\ k_2, k_3\geq 1} } { {2k} \choose { k_1 , k_2, k_3 } } 
(  2| \log a |)^{k_1}  \Bigg(  \sum_{\ell = 1}^\infty   \frac{( \log \ell)^{k_2 +k_3-2}  }{ \ell^2}  \Bigg) .
\ee
The function $ f(x) = ( \log x)^{k_2 + k_3 - 2} x^{-2}$ has a maximum at $ x = x_0 = e^{ (k_2 + k_3 - 2)/2 } $ and 
it  is increasing  for $ 0 < x < x_0$ and decreasing for $ x > x_0$. Thus
\begin{align*}
 \sum_{\ell = 1}^\infty    \frac{   ( \log \ell)^{k_2 +k_3-2}}{ \ell^2} & \ll x_0 \frac{ ( \log x_0 )^{k_2 + k_3 - 2 }}{ x_0 ^2 } + \int_{x_0}^\infty   \frac{  ( \log u )^{ k_2 + k_3 - 2 }}{ u^2 } du \\
 & \ll (C k)^{k_2 + k_3 -2} + \int_{\log x_0 }^\infty v^{k_2 + k_3 - 2 } e^{-v} dv \\
 & \ll (C k)^{k_2 + k_3 -2} +\Gamma( k_2 + k_3 - 1)   
   \ll (C k)^{k_2 + k_3 - \frac12  }.
\end{align*}
Hence,
\begin{align*}
   \frac{1}{ 2 \pi} \int_0^{2 \pi} \big(   \log  | a- be^{iv}  | \big)^{2k} dv  
    \ll & |   \log a|^{2k} + k^{7/2}  \sum_{\substack{k_1 + k_2 + k_3 = 2k  \\ k_2 , k_3 > 0 }} { {2k} \choose { k_1 , k_2, k_3 } } | 2 \log a | ^{k_1}  (Ck)^{k_2 + k_3 }\\
    \ll & | \log a|^{2k} + k^{7/2} ( | 2 \log a | +  Ck )^{2k}  \\
   \ll &(C | \log a |)^{2k} +  (Ck)^{2k}     .
\end{align*}

Now  consider the case $ a=b>0$. We have
\begin{align*}
   \frac{1}{ 2 \pi} \int_0^{2 \pi} \big(   \log  | a- ae^{iv}  | \big)^{2k} dv  
   = &   \frac{1}{ 2 \pi} \int_0^{2 \pi} \big( \log |2a| +  \log  | \sin v/2  | \big)^{2k} dv
 \\ 
   \leq & 2^{2k-1} \bigg(  (\log |2a| )^{2k}+  \frac{1}{ \pi}   \int_0^{ \pi} \big(   \log  | \sin v/2  | \big)^{2k} dv \bigg) .
  \end{align*}
Note that for $0\leq v\leq \pi$
we have \, $v/4\leq  |\sin(v/2)|  \leq 1$.   Thus, the last line  is
\begin{align*}   
   \leq& 2^{2k-1} \bigg(  (\log |2a| )^{2k}+  \frac{1}{ \pi}   \int_0^{ \pi} \big(   \log  (v/4 )\big)^{2k} dv \bigg)  \\
   \ll &  C^{2k}\bigg( | \log a |^{2k}   +       \int_0^1 \big(   \log  x  \big)^{2k} dx\bigg)  
    \ll    C^{2k}\bigg( | \log a |^{2k}   +      \Gamma(2k+1) \bigg)  \\
   \ll &(C | \log a |)^{2k} +  (Ck)^{2k}  .
       \end{align*}
\end{proof}


\begin{lemma}\label{lem : Dir Poly approx}
Let $L(s)=L(s, \chi)$ be a Hecke $L$-function  attached to an ideal class character of the quadratic field $\mathbb{Q}(\sqrt{D})$. For $\s > 1 $  write 
$$ 
\log L(s)  = \sum_{p, n } \frac{a(p^n)}{p^{n s}} ,
$$
and for $Y\geq 2$ and any $s$ let
$$
R_Y(s)= \sum_{p^n\leq Y } \frac{ a(p^n)}{p^{n s}} .
$$
Suppose that \  $ 1/2 < \sigma < 1 $ and   $B_1 > 0$ are fixed, and 
that $ Y = ( \log T)^{B_2}$ with $B_2 >2(B_1+1)/(\sigma-1/2)$.
Then
$$ 
\log L(s) = R_Y (s) + O\big( ( \log T)^{-B_1} \big)
$$
for all $ t \in [T, 2T]$ except on a set of measure $\ll T^{1-d(\sigma)}$, where $ d(\sigma )>0$.   \end{lemma}

\begin{proof} 
Using an  approximate functional equation for $L(s)$ (for example, see Section A.12 of \cite{KV})  in a standard way, we find that
$$ \int_T^{2T} |L(1/2 +it) |^2 dt \ll T (\log T)^4 . $$
From this and  Theorem 1 of \cite{Le1} we obtain the zero-density estimate
$$ N_L( \sigma, T, 2T):=\sum_{\substack{T< \gamma\leq 2T \\ \beta\geq \sigma}}
 1\ll T^{1-a_1(\sigma-1/2)} (\log T)^{12} $$
uniformly for $ \sigma \geq 1/2$, where $\rho=\beta+i\gamma$   denotes a generic nontrivial zero of $L(s)$ and
 $a_1 >0$ is a constant independent of $\sigma$. 

 Now let $ s=\sigma+it$ with $ 1/2<\sigma <1 $ and $ T \leq t \leq 2T$. By Perron's formula (see Titchmarsh~\cite{T}, pp.60--61)
 \be\label{Perron 1}
 R_Y(s)= \sum_{ p^n \leq Y } \frac{ a(p^n)}{p^{ns}} = \frac{1}{2 \pi i } \int_{c-iY}^{c+iY} \log L( s+ w) \frac{Y^w}{w} dw +O( Y^{-\sigma + \epsilon}),
 \ee
where $c = 1-\sigma+\epsilon$ with  $0< \epsilon <1/4$. 
Let $w_0= (1/2-\sigma)/2$ and assume 
 that $L(s+w )$ has no zeros in the half-strip  given by
 $ \Re w \geq  \frac34(1/2-\sigma) $, $ | \Im w | \leq Y+1$. 
 Then in the slightly smaller half-strip  $\Re w\geq w_0, |\Im w|\leq Y$ we have
\be\label{L'/L 1}
 \frac{L'}{L}(s+w) \ll \log T
\ee
 (see Iwaniec and Kowalski~\cite{IK},  Proposition 5.7). 
 Observe that this holds  for all $ t \in [T,2T]$ except for $t$ in a set of measure 
  $$ \ll (2Y+2) \cdot N(\tfrac12+\tfrac14(\sigma-\tfrac12) , T, 2T) \ll  T^{1- a_1(\sigma-1/2)/4} (\log T)^{B_2 + 12}.  $$
Now, integrating \eqref{L'/L 1} along the horizontal segment from $w$ to $w+2$, we see that
$$ \log L(s+w) = O( \log T ) .$$ 
Using this and shifting the contour to the left in \eqref{Perron 1}, we obtain
\begin{align*}
 R_Y(s) &= \log L(s) +  \frac{1}{2 \pi i } \int_{w_0-iY}^{w_0+iY} \log L( s+ w) \frac{Y^w}{w} dw +O( Y^{-\sigma + \epsilon}) \\
 &= \log L(s) + O\big( ( \log T)^{ 1- B_2( \sigma-1/2) /2 } + (\log T)^{(-\sigma+\epsilon)B_2 }\big) .
 \end{align*} 
 Therefore, 
 \be\label{log approx 1}
  \log L(s) = R_Y (s) + O\big( ( \log T)^{ 1-  B_2( \sigma-1/2) /2 } + (\log T)^{-(\sigma-\epsilon)B_2 }\big)
  \ee
 holds for all $ t \in [T,2T]$ except for a set of measure 
  $$ \ll   T^{1-c_1 (\sigma-1/2)/4 }(\log T)^{B_2 + 12}.  $$
 Given $B_1>0$, if we     take   $B_2> 2(B_1+1)/(\sigma-1/2)$, both 
  error terms in \eqref{log approx 1} are $O((\log T)^{-B_1})$.
This proves the lemma.
\end{proof}


\begin{lemma}\label{lem:dir poly mmt 1}
Let $ 2  \leq y \leq z $ and let $k$ be a positive integer $ \leq \log T/(3 \log z)$. Suppose that $ |a(p) | \leq 2 $. Then
$$ \frac{1}{T} \int_T^{2T} \bigg|  \sum_{ y < p \leq z } \frac{ a(p)}{ p^{\sigma +it}} \bigg|^{2k} dt \ll 2^{2k} k! \bigg(  \sum_{ y < p \leq z } \frac{1}{ p^{2 \sigma}} \bigg)^k + 2^{2k} T^{- 1/3} $$
and
$$ \E \bigg( \bigg|  \sum_{ y < p \leq z } \frac{ a(p)X(p)}{ p^{\sigma }} \bigg|^{2k}   \bigg)   \ll 2^{2k} k! \bigg(  \sum_{ y < p \leq z } \frac{1}{ p^{2 \sigma}} \bigg)^k . $$
\end{lemma}
This is a simple modification of Lemma 3.2 of \cite{LLR} so we omit the proof.

\begin{lemma}\label{lem:dir poly mmt 2}
Let $R_{j,Y}(s)$  be the Dirichlet polynomial approximation   corresponding to  $\log L_j(s)$ in Lemma~\ref{lem : Dir Poly approx} and
let $R_{j,Y}(\s, X)$ be the analogous expression for  $\log L_j(\s, X)$.
Let $ 1/2 < \sigma < 1 $ and $Y= ( \log T)^{B_2} $, where
 $B_2$ is as in Lemma~\ref{lem : Dir Poly approx}.
Then for any positive integers $k \leq \log T/(3 \log Y) $ and $ j \leq J$, we have
$$ \frac{1}{T} \int_T^{2T} | R_{j,Y} (\sigma+it) |^{2k} dt \ll \bigg( \frac{ C k^{1-\sigma}}{ ( \log k)^\sigma } \bigg)^{2k} $$
and 
$$ \E \big(| R_{j,Y} (\sigma,X) |^{2k} \big) \ll \bigg( \frac{ C k^{1-\sigma}}{ ( \log k)^\sigma}  \bigg)^{2k}. $$
Here $C$ is a constant depending only on $\sigma $.
\end{lemma}
\begin{proof}
To prove the first estimate it is enough to show that
$$ \frac{1}{T} \int_T^{2T} \bigg| \sum_{ p \leq Y}   \frac{ a_j (p)}{p^{\sigma + it}}\bigg|^{2k}  dt \ll \bigg( \frac{ C k^{1-\sigma}}{ ( \log k)^\sigma } \bigg)^{2k} .$$
By Lemma \ref{lem:dir poly mmt 1} and the prime number theorem
\begin{align*}
\frac{1}{T} \int_T^{2T} \bigg| \sum_{ p \leq Y}   \frac{ a_j (p)}{p^{\sigma + it}}\bigg|^{2k}  dt & \leq \frac{2^{2k-1}}{T} \int_T^{2T} \bigg| \sum_{ p \leq k \log k}   \frac{ a_j (p)}{p^{\sigma + it}}\bigg|^{2k}  dt + \frac{2^{2k-1}}{T} \int_T^{2T} \bigg| \sum_{ k \log k < p \leq Y}   \frac{ a_j (p)}{p^{\sigma + it}}\bigg|^{2k}  dt \\
& \ll 2^{2k-1} \bigg( \sum_{p \leq k \log k } \frac{2}{ p^\sigma } \bigg)^{2k} +  2^{4k-1} k! \bigg(  \sum_{ k \log k < p \leq Y } \frac{1}{ p^{2 \sigma}} \bigg)^k + 2^{4k-1} T^{- 1/3}\\
& \ll  \bigg( \frac{ C k^{1-\sigma}}{ ( \log k)^\sigma}  \bigg)^{2k}.
\end{align*}

The estimate for the expectation may be treated similarly.

\end{proof}

\begin{lemma}\label{lem:dir poly mmt 3}
Let
$$ Q_{j,Y} ( s) = \sum_{ n \leq Y} \frac{ b_j (n)}{ n^s}
\qquad \hbox{and} \qquad  
Q_{j,Y} ( \s, X) = \sum_{ n \leq Y} \frac{ b_j (n)X(n)}{ n^\s},
$$
where for each $j\leq J$,  the  $b_j(n)$ are bounded complex numbers, 
and, if $n=\prod_p p^\alpha$, then $X(n) =\prod_p X(p)^\alpha$ with the $X(p)$ 
uniformly and independently distributed on $\mathbb T$.
Let $k_j, k_j'$ for $ j \leq J$ be positive integers and $k = \sum_{j\leq J} k_j$, $k' = \sum_{j\leq J} k_j'$. 
Then
\be\notag 
\begin{split}
\frac1T \int_T^{2T} \prod_{ j \leq J} \bigg(  Q_{j,Y}( \sigma+it)^{k_j}  \overline{Q_{j,Y} ( \sigma +it)}^{k_j'}       \bigg) dt 
= &\E \bigg(\prod_{ j \leq J} \bigg(  Q_{j,Y}( \sigma,X)^{k_j}  \overline{Q_{j,Y} ( \sigma ,X)}^{k_j'}       \bigg)  \bigg)\\
&\qquad+ O \bigg( \frac{ (CY^{2-\sigma})^{k+ k'}}{T}   \bigg).
\end{split}
\ee
\end{lemma}
\begin{proof}
\begin{align*}
  \frac1T \int_T^{2T} \prod_{ j \leq J}& \bigg(  Q_{j,Y}( \sigma+it)^{k_j}  \overline{Q_{j,Y} ( \sigma +it)}^{k_j'}       \bigg) dt  \\
\qquad \qquad& =  \frac1T \int_T^{2T} \bigg( \sum_{n_{i,j} \leq Y}    \frac{ b_1 (n_{1,1})   \cdots  
b_1 ( n_{k_1 , 1} ) }{(n_{1,1} \cdots n_{ k_1 , 1 } )^{\sigma +it} } \   \cdots    \  \frac{ b_J (n_{1,J})     \cdots    b_J ( n_{k_J , J} ) }{(n_{1,J} \cdots n_{ k_J , J } )^{\sigma +it} }\bigg) \\
& \qquad \qquad\qquad \cdot \bigg(\sum_{m_{i,j} \leq Y}    \frac{ \overline{ b_1 (m_{1,1})} \cdots \overline{b_1 ( m_{k_1' , 1} ) }}{(m_{1,1} \cdots m_{ k_1' , 1 } )^{\sigma-it} } \  \cdots    \  \frac{ \overline{b_J (m_{1,J})} \cdots  \overline{b_J ( m_{k_J' , J} )} }{(m_{1,J} \cdots m_{ k_J' , J } )^{\sigma -it} }\bigg)\;\; dt .
\end{align*}
The contribution of the diagonal terms to this is
\begin{align*}
\mathcal{ D}=& \sum_{\substack{ n_{i,j}, m_{i,j} \leq Y \\   \prod n_{i,j} = \prod m_{i,j}   } }     \frac{ b_1 (n_{1,1}) \cdots b_1 ( n_{k_1 , 1} ) }{(n_{1,1} \cdots n_{ k_1 , 1 } )^{\sigma  } } \cdots      \frac{ b_J (n_{1,J}) \cdots b_J ( n_{k_J , J} ) }{(n_{1,J} \cdots n_{ k_J , J } )^{\sigma } } \\
& \qquad\qquad \cdot     \frac{ \overline{ b_1 (m_{1,1})} \cdots \overline{b_1 ( m_{k_1' , 1} ) }}{(m_{1,1} \cdots m_{ k_1' , 1 } )^{\sigma} } \cdots      \frac{ \overline{b_J (m_{1,J})} \cdots  \overline{b_J ( m_{k_J' , J} )} }{(m_{1,J} \cdots m_{ k_J' , J } )^{\sigma } } \\
 =& \E \bigg( \prod_{ j \leq J} \bigg(  Q_{j,Y}( \sigma,X)^{k_j}  \overline{Q_{j,Y} ( \sigma ,X)}^{k_j'}       \bigg)  \bigg).
\end{align*}
The off-diagonal contribution is
\begin{align*}
\mathcal{O}   = &  \sum_{\substack{ n_{i,j}, m_{i,j} \leq Y \\   \prod n_{i,j} \neq \prod m_{i,j}   } }    \frac{ b_1 (n_{1,1}) \cdots b_1 ( n_{k_1 , 1} ) }{(n_{1,1} \cdots n_{ k_1 , 1 } )^{\sigma  } } \;   \cdots   \;   \frac{ b_J (n_{1,J}) \cdots b_J ( n_{k_J , J} ) }{(n_{1,J} \cdots n_{ k_J , J } )^{\sigma } } \\
& \qquad \qquad  \cdot     \frac{ \overline{ b_1 (m_{1,1})} \cdots \overline{b_1 ( m_{k_1' , 1} ) }}{(m_{1,1} \cdots m_{ k_1' , 1 } )^{\sigma} } \;   \cdots     \; \frac{ \overline{b_J (m_{1,J})} \cdots  \overline{b_J ( m_{k_J' , J} )} }{(m_{1,J} \cdots m_{ k_J' , J } )^{\sigma } } \
    \cdot \   \bigg(\frac{(m/n)^{2iT} - (m/n)^{iT} }{i T \log (m/n)}\bigg),
\end{align*}
where $n =\prod n_{i,j} $ and $ m =\prod m_{i,j} $. Since $ n, m \leq Y^{k+k'} $ and $n\neq m $,
$$ \frac{1}{|\log (m/n)| }  \ll Y^{k+ k'} . 
$$ 
Hence,
\begin{align*}
\mathcal{O} &   \ll  \frac{ Y^{k+ k'}}{T}   \sum_{\substack{ n_{i,j}, m_{i,j} \leq Y \\   \prod n_{i,j} \neq \prod m_{i,j}   } }    \frac{  | b_1 (n_{1,1}) \cdots b_1 ( n_{k_1 , 1} ) | }{(n_{1,1} \cdots n_{ k_1 , 1 } )^{\sigma  } } \cdots      \frac{| b_J (n_{1,J}) \cdots b_J ( n_{k_J , J} ) |  }{(n_{1,J} \cdots n_{ k_J , J } )^{\sigma } } \\
 & \hskip1.5in   \cdot     \frac{ |b_1 (m_{1,1}) \cdots b_1 ( m_{k_1' , 1} )  | }{(m_{1,1} \cdots m_{ k_1' , 1 } )^{\sigma} } \cdots      \frac{ |b_J (m_{1,J}) \cdots  b_J ( m_{k_J' , J} )|  }{(m_{1,J} \cdots m_{ k_J' , J } )^{\sigma } } \\
 &   \ll  \frac{ (CY)^{k+ k'}}{T}   \sum_{ n_{i,j}, m_{i,j} \leq Y  }    \frac{  1}{(n_{1,1} \cdots n_{ k_1 , 1 } )^{\sigma  } } \cdots      \frac{1  }{(n_{1,J} \cdots n_{ k_J , J } )^{\sigma } } \\
& \hskip1.5in  \cdot     \frac{ 1}{(m_{1,1} \cdots m_{ k_1' , 1 } )^{\sigma} } \cdots      \frac{ 1}{(m_{1,J} \cdots m_{ k_J' , J } )^{\sigma } } \\
&   \ll  \frac{ (CY^{2-\sigma})^{k+ k'}}{T}.
\end{align*}
The lemma now follows on combining our estimates for the diagonal and off-diagonal terms $\mathcal{D}$ and $\mathcal{O}$.
\end{proof}


\section{Lemmas on moments of logarithms of $L$-functions}
%


We will  frequently appeal to  the following  three lemmas.  
 
 \begin{lemma} \label{lem: 2kth mmt real}
Let $ 1/2 < \sigma \leq 2 $ be fixed. There exists a  constant $C>0$ depending at most on $J$,  such that for every positive integer $k$ we have
$$ \frac1T  \int_T^{2T} \bigg|\;  \log  \bigg| \sum_{j \leq J} c_j L_j ( \sigma+it) \bigg| \; \bigg|^{2k}  dt \ll  (Ck )^{4k} . $$
\end{lemma}


\begin{lemma}\label{lem: 2kth mmt imag}
Let $ 1/2 < \sigma \leq 2 $ be fixed. There exist an absolute  constant $C_1>0$ and a constant $C_2>0$ depending on $ \sigma$  such that for every positive integer $k \leq \log T / ( C_2 \log \log T)$, we have 
$$
 \frac1T  \int_T^{2T} | \log c_j L_j(s)  |^{2k}  dt  \ll  (C_1k )^{2k}
 $$
 and
 $$ 
 \frac1T  \int_T^{2T}   \big|   \log{c_i L_i(s)} -\log{c_j L_j(s)}  \big|^{2k}  dt 
 \ll  (C_1k )^{2k}.
$$
\end{lemma}

 \begin{lemma} \label{2kth mmt random}
Let $ 1/2 < \sigma \leq 2 $ be fixed. For every integer $k > 0$ we have
$$  \E \bigg(     \bigg|\,\log  \bigg| \sum_{j \leq J} c_j L_j ( \sigma, X) \bigg|\,\bigg|^{2k} \bigg)   \ll (Ck)^{2k}
$$
and 
$$  \E \bigg(     \bigg|\,\log    c_j L_j ( \sigma, X)  \bigg|^{2k} \bigg)   \ll (Ck)^{k},
$$
where,  in either case, $C>0$ is a  constant depending  at most on $J$.
\end{lemma}

We  will sometimes  use   Lemmas~\ref{lem: 2kth mmt real} and \ref{lem: 2kth mmt imag} 
to show that we may restrict certain sets to lie within   intervals of length 
$\  \approx \sL=\log\log T$ at the cost of a small error.  
Here is a typical example.
Let $\B$ be a Borel set in $\mathbb R$ and let 
 $A$  be a  positive constant. Set
$$ I_T=(-A\sL , A\sL ].$$
Then 
\be\notag
\begin{split}
\mathrm{meas}\big\{ t\in[T, 2T] : |\log |c_iL_i(s)|| \notin I_T \big\}
 \leq & \int_{T}^{2T} \bigg(\frac{|\log |c_iL_i(s)||}{A\sL}\bigg)^{2k} dt  
 \ll   \;T\bigg(\frac{ Ck}{ A \sL }\bigg)^{2k}.
\end{split}
\ee
Taking $k=\sL$ and $A$ sufficiently large relative to $C$, we see that this is 
$\ll T (\log T)^{-B}$, where $B>0$ is an arbitrarily large   constant. Thus,
\be\notag
\begin{split}
\mathrm{meas}  \big\{ t\in[T, 2T] : \, | \log |c_iL_i(s)| \,| \in \B  \big\}  
=&\mathrm{meas}\big\{ t\in[T, 2T] : |\log |c_iL_i(s)|\,| \in \B \cap  I_T \big\} \\
&\quad +O\big(T (\log T)^{-B} \big).
\end{split}
\ee
When required, we will restrict sets in this way by simply  writing ``by 
Lemma~\ref{lem: 2kth mmt real}  (or \ref{lem: 2kth mmt imag})''.   

The proof of Lemma~\ref{lem: 2kth mmt real} is a straightforward modification of  the proof of Proposition 2.5 in \cite{LLR} 
(one just replaces  $\zeta(s)-a$ by $E(s, Q)$ throughout), so we do not include it.

 \subsection{Proof of Lemma~\ref{lem: 2kth mmt imag}}



 A little thought shows that it is enough to prove that
$$ \frac 1T \int_T^{2T} |   \log L_j ( \sigma + it ) |^{2k} dt \ll (Ck)^{2k}. $$
Let $\mathscr A (T) = \mathscr A_\sigma ( T) $ be the set of $ t \in [T, 2T]$ such that
$$ \log L ( \sigma + it ) = R_Y  ( \sigma + it ) + O ( ( \log T)^{- B_1 })  $$
and let $ \mathscr A'(T)  = [T, 2T] \setminus \mathscr A(T)$. Then by Lemma \ref{lem : Dir Poly approx} 
\begin{equation}\label{eqn A'(T) error}
 \mathrm{meas}( \mathscr A'(T) ) \ll T^{1- d(\sigma)}.
\end{equation}
Splitting the integral as
\be\label{split int}
 \frac1T \int_T^{2T} | \log L( \sigma  + it)|^{2k} dt 
= \bigg\{\int_{\mathscr A(T)} + \int_{ \mathscr A'(T)}\bigg\} | \log L( \sigma  + it)|^{2k} dt, 
\ee
we first estimate the integral over $ \mathscr A'(T)$. We have
\be\label{int A'} 
\begin{split}
\frac1T  \int_{\mathscr A'(T)}& | \log L( \sigma  + it)|^{2k} dt \\
& \leq  2^{k-1} \bigg( \frac1T \int_{\mathscr A'(T)} | \log | L( \sigma  + it)| |^{2k} dt
+  \frac1T \int_{\mathscr A'(T)} | \arg L( \sigma  + it)|^{2k} dt  \bigg) . 
\end{split}
\ee
By Lemma 9.4 in \cite{T} 
$$ \arg L( \sigma + it) = O( \log T),$$
so   
$$  \frac1T \int_{\mathscr A'(T)} | \arg L( \sigma  + it)|^{2k} dt     \ll  T^{- d(\sigma)} (C \log T)^{2k} . $$
By the Cauchy-Schwartz inequality,
Lemma \ref{lem: 2kth mmt real}, and \eqref{eqn A'(T) error}, the second integral in   
\eqref{int A'} is
\be\notag 
\begin{split}
 \frac1T \int_{\mathscr A'(T)} | \log | L( \sigma  + it)||^{2k} dt     \ll & T^{ - d(\sigma)/2}  \bigg(\frac 1T \int_T^{2T}  | \log | L( \sigma  + it)||^{4k} dt \bigg)^{1/2} \\
   \ll  &T^{ - d(\sigma)/2}  (Ck)^{4k} . 
 \end{split}
\ee
Thus,
\be\label{int A' final}
 \frac1T \int_{\mathscr A'(T)} | \log L( \sigma  + it)|^{2k} dt \ll  T^{- d(\sigma)} ( C \log T)^{2k}  +  T^{ - d(\sigma)/2}  (Ck)^{4k}. 
 \ee

Next we estimate the integral over $\mathscr A(T)$. By the definition of $ \mathscr A(T)$
\begin{align*} 
 \frac1T \int_{\mathscr A(T)} | \log L( \sigma  + it)|^{2k} dt   & =  \frac1T \int_{\mathscr A(T)} | R_Y ( \sigma  + it) + O( ( \log T)^{- B_1 }) |^{2k} dt  \\
& \leq  \frac1T \int_T^{2T} \bigg| \sum_{ p \leq Y} \frac{ a(p)}{ p^{\sigma + it}} + O(1) \bigg|^{2k} dt   \\
& \leq 2^{2k-1} \bigg(   \frac1T \int_T^{2T} \bigg| \sum_{ p \leq Y} \frac{ a(p)}{ p^{\sigma + it}}  \bigg|^{2k} dt   + C^k \bigg)  .  
\end{align*}
Defining
$$ a_k (n) = \sum_{\substack{ p_1 \cdots p_k = n \\   p_i \leq Y }} 
a(p_1 ) \cdots a(p_k) , $$
we find that  since $| a(p) | \leq 2 $, 
$$ | a_k ( n) | \leq   \sum_{ p_1 \cdots p_k = n } 2^k \leq 2^k k! \ll (Ck)^k . $$
Thus, for $ Y^k \ll T $ and $1/2<\s\leq 2$   fixed,
\begin{align*}  
   \frac1T \int_T^{2T} \bigg| \sum_{ p \leq Y} \frac{ a(p)}{ p^{\sigma + it}}  \bigg|^{2k} dt & =     \frac1T \int_T^{2T} \bigg| \sum_{ n \leq Y^k} \frac{ a_k (n)}{ n^{\sigma + it}}  \bigg|^2 dt \\
& = \frac1T ( T+O(Y^k )) \sum_ { n \leq Y^k } \frac{ |a_k (n) |^2 }{ n^{2 \sigma}} \\
& \ll  \sum_{n=1}^\infty \frac{   ( Ck)^{2k}}{ n^{2 \sigma}}  
 \ll (Ck)^{2k}.
   \end{align*}
 Since $Y = ( \log T)^{B_2}$, this inequality holds for $k \leq \log T/( B_2 \log \log T) $.

Combining this with  \eqref{split int} and \eqref{int A' final}, we obtain
$$ \frac1T \int_T^{2T}  | \log L( \sigma  + it)|^{2k} dt \ll  (Ck)^{2k}+  T^{- d(\sigma)} ( C \log T)^{2k}  +  T^{ - d(\sigma)/2}  (Ck)^{4k} 
$$
for $k \leq \log T/( B_2 \log \log T) $. In the second and third error terms we have
$$ (C \log T)^{2k} \leq C^{2k} T^{ 2  /B_2} \qquad \mathrm{and} \qquad (Ck)^{4k} \leq T^{4/B_2 }. $$
Thus, choosing  $  B_2 $  large enough, we  find that
$$ \frac1T \int_T^{2T}  | \log L( \sigma  + it)|^{2k} dt \ll  (Ck)^{2k} $$
for $k \leq \log T/( B_2 \log \log T) $, as required.
 

 
\subsection{Proof of Lemma~\ref{2kth mmt random}}

We define a measure on Borel sets $\B \in \mathbb{R}^{2J}$ by
$$\Phi (\B) := \P ( \L_0  ( \sigma, X) \in \B ) ,$$
where
$$ \L_0 ( \sigma, X) :=  \big( \log | L_1 ( \sigma, X)| , \dots, \log |L_J ( \sigma, X) | , 
\arg   L_1 ( \sigma, X) , \dots, \arg   L_J ( \sigma, X)  \big) . $$
By a straightforward generalization of the proofs of Theorems 5 and 6 in Borchsenius and Jessen~\cite{BJ}, one can 
show that $\Phi $ is absolutely continuous and that  its density function $ F( \u, \v) $ satisfies 
$$ 
F( \u, \v)\ll \exp \big(-A( u_1^2 + \cdots + u_J ^2 +  v_1 ^2 + \cdots +  v_J ^2 )  \big) 
$$
for some constant $A>0$, where $ \u = ( u_1 , \dots, u_J ) $ and $ \v = ( v_1 , \dots, v_J ) $.
Hence
\begin{align*}
\E \bigg(  \bigg| \log \bigg| \sum_{j=1}^J c_j L_j ( \sigma, X) \bigg| \ \bigg|^{2k} \bigg) 
&=  \int_{-\infty}^\infty  \cdots \int_{-\infty}^\infty  
\bigg| \log \bigg| \sum_{j=1}^J c_j e^{u_j + i v_j } \bigg| \  \bigg|^{2k} F( \u, \v) d\u d\v \\
& \ll  \int_{-\infty}^\infty  \cdots \int_{-\infty}^\infty  
\bigg| \log \bigg| \sum_{j=1}^J c_j e^{u_j + i v_j } \bigg|\  \bigg|^{2k} e^{ - A( u_1^2 + \cdots + u_J ^2 +  v_1 ^2 + \cdots +  v_J ^2 )  }  d\u d\v .
\end{align*}
Since $e^{iv_j}$ is periodic,  the integral with respect to   $v_j$  is of the form
\begin{align*} 
\sum_{m= -\infty}^\infty \int_0^{2\pi} \bigg| \log \bigg|  \sum_{j=1}^J c_j e^{u_j + i v_j } \bigg|\  \bigg|^{2k}&e^{-A ( v_j +2 \pi m )^2 }dv_j  \\
=&  \int_0^{2\pi}  \bigg| \log \bigg| \sum_{j=1}^J c_j e^{u_j + i v_j } \bigg|\  \bigg|^{2k}\sum_{m= -\infty}^\infty   e^{-A ( v_j +2 \pi m )^2 }dv_j  \\
\ll &\int_0^{2\pi} \bigg| \log \bigg| \sum_{j=1}^J c_j e^{u_j + i v_j } \bigg|\  \bigg|^{2k}\ dv_j   .
\end{align*}
Thus,
\begin{align}\label{exp lin comb 1}
\E \bigg(  \bigg|  \log \bigg|  \sum_{j=1}^J &  c_j L_j ( \sigma, X) \bigg| \  \bigg|^{2k} \bigg)   \\
& \ll  \int_{-\infty}^\infty  \cdots \int_{-\infty}^\infty  \bigg( \int_0^{2\pi} \cdots \int_0^{2 \pi}   \bigg|
 \log \bigg| \sum_{j=1}^J c_j e^{u_j + i v_j } \bigg| \ \bigg|^{2k} d\v \bigg)  
 e^{ - A( u_1^2 + \cdots + u_J ^2 )}      d\u   .   \notag
\end{align}


We first integrate with respect to $u_1$ and $v_1$ and see that   
\begin{align*} \int_{-\infty}^\infty  \int_0^{2 \pi}  
 \bigg| \log \bigg| \sum_{j=1}^J c_j e^{u_j + i v_j } \bigg|\bigg|^{2k}     
 e^{ - A u_1 ^2 } dv_1 du_1 
 =& \int_{-\infty}^\infty  \int_0^{2 \pi}  
 \big| \log  | B+  c_1 e^{u_1 + i v_1 }  |\big|^{2k}     e^{ - A u_1 ^2 } dv_1 du_1 ,
 \end{align*}
where $ B =  \sum_{j=2}^J c_j e^{u_j + i v_j }$. Dividing the $u_1$-integral into two pieces, we see that this equals
\be\bigg(\int_{ |B| \ge |c_1 | e^{u_1}  } +\int_{ |B| < |c_1 | e^{u_1}}  \bigg)
\bigg(  \int_0^{2 \pi}   \big| \log   | B+  c_1 e^{u_1 + i v_1 } |\big|^{2k}  dv_1 \bigg)  e^{ - A u_1 ^2 }  du_1. 
\ee
By Lemma~\ref{lem: log 2k dv} the first integral is 
$$
\ll   \int_{ |B| \geq | c_1 | e^{u_1  } } \big(  ( C\log |B|)^{2k}  + (Ck)^{2k} \big)  e^{ - A u_1^2 }  d u_1 
\ll ( C\log |B|)^{2k}  + (Ck)^{2k} ) .
$$
Also by Lemma~\ref{lem: log 2k dv}, the second integral is
\begin{align*}
\int_{ |B| < |c_1 | e^{u_1}} \bigg(\int_0^{2 \pi}    
&\big| \log   |  c_1 e^{u_1}  + B e^{-i v_1 } |\big|^{2k}  dv_1   \bigg) 
e^{ - A u_1 ^2 }  du_1  \\
\ll &\int_{ |B| < |c_1 | e^{u_1}}\big( (C|\log c_1 e^{u_1}|)^{2k} +(Ck)^{2k}  \big)
e^{ - A u_1 ^2 }  du_1 \\
\ll &\int_{-\infty}^\infty   (C+ | u_1  | )^{2k}   e^{ - A u_1 ^2 }  du_1 +  (Ck)^{2k}\\
\ll & C^{2k}  \int_{-\infty}^\infty   u_1^{2k}   e^{ - A u_1 ^2 }  du_1 +  (Ck)^{2k} 
\ll   C^{2k} \Gamma(k+\tfrac12) +  (Ck)^{2k} \\
\ll& (Ck)^{2k}.
\end{align*}
Hence
\begin{align*} \int_{-\infty}^\infty  \int_0^{2 \pi}  
 \bigg| \log \bigg| \sum_{j=1}^J c_j e^{u_j + i v_j } \bigg|\bigg|^{2k}     
 e^{ - A u_1 ^2 } dv_1 du_1 
 \ll&  \bigg(C  \bigg |\log  \bigg| \sum_{j=2}^J c_j e^{u_j + i v_j } \bigg| \bigg)^{2k}    +(Ck)^{2k}
 \end{align*}
 
Proceeding inductively, we find that
\begin{align} 
  \int_{-\infty}^\infty  \cdots \int_{-\infty}^\infty  \bigg( \int_0^{2\pi} \cdots \int_0^{2 \pi}   \bigg|
 \log \bigg| \sum_{j=1}^J c_j e^{u_j + i v_j } \bigg| \ \bigg|^{2k} d\v \bigg)  
 e^{ - A( u_1^2 + \cdots + u_J ^2 )}      d\u  
 \ll  (Ck)^{2k} .   \notag
\end{align}
Thus, by \eqref{exp lin comb 1}
$$
\E \bigg(  \bigg|  \log \bigg|  \sum_{j=1}^J    c_j L_j ( \sigma, X) \bigg| \  \bigg|^{2k} \bigg)  
\ll (Ck)^{2k}.
$$

We now prove the second inequality of the lemma, namely,
$$  \E \bigg(     \bigg|\,\log    c_j L_j ( \sigma, X)  \bigg|^{2k} \bigg)   \ll (Ck)^{k}.
$$
as in  the proof for the first inequality, we have
\begin{align*}
  \E \bigg(     \bigg|\,\log    L_j ( \sigma, X)  \bigg|^{2k} \bigg)  &  \ll  \int_{-\infty}^\infty \int_{-\infty}^\infty | u+iv|^{2k} e^{-A (u^2 + v^2 )} du dv \\
  &  =  \int_{-\infty}^\infty \int_{-\infty}^\infty ( u^2 + v^2 )^k e^{-A (u^2 + v^2 )} du dv \\
  & =  \int_0^{2 \pi} \int_0^\infty r^{2k} e^{-Ar^2} r dr d\theta \\
  & \ll (Ck)^k . 
\end{align*}

\section{Proof of Theorem~\ref{lem: discrepancy}  }

\subsection{Two lemmas pertaining to the proof of Theorem~\ref{lem: discrepancy} }

To prove Theorem~\ref{lem: discrepancy} we also require  the following two lemmas.

\begin{lemma}\label{lem: transforms}
Let  $\u = ( u_1 , \dots, u_J)$ and similarly $\v, \w,$ and $\z$ be vectors in $\mathbb R^{J}$.
Let 
$$ (\w, \z)\cdot (\u,\v)  = \sum_{j \leq J} (w_j u_j + z_j v_j ),$$
and define
$$ \widehat{\Psi}_T ( \w, \z) := \int_{ \mathbb{R}^{2J}} e^{ i  (\w, \z)\cdot (\u,\v) } d\Psi_T ( \u, \v)$$
and
$$ \widehat{\Psi} ( \w,\z ) := \int_{ \mathbb{R}^{2J} }  e^{ i  (\w, \z)\cdot (\u,\v) } d\Psi  ( \u, \v).
$$
If $A >1$, there exists a constant   $b=b(\s, A)>0$ such that for all $ \w$ and $ \z $ with $ |w_j|, |z_j| \leq b (\log T)^\sigma $, we have
\be\notag
 \widehat{\Psi}_T ( \w, \z) - \widehat{\Psi} (\w,\z) = O \big( ( \log T)^{-A} \big).
 \ee
\end{lemma}

\begin{proof} 
By the definition  of $\Psi_T $ and $\Psi$ (see \eqref{Psi_T} and \eqref{Psi}),  we may write  
\begin{equation}\notag
\begin{split}
\widehat{\Psi}_T ( \w, \z) = \frac1T \int_T^{2T} 
\exp \Bigg[ i \bigg(& \sum_{ j <  J } w_j \log \bigg| \frac{c_j }{c_J} \frac{ L_j }{L_J} ( \sigma +it)  \bigg| + w_J \log |c_J L_J ( \sigma+it) |   \\
&  + \sum_{ j < J } z_j \arg \frac{c_j}{c_J} \frac{L_j}{L_J} ( \sigma + it)   + z_J \arg c_J L_J ( \sigma +it) \bigg) \Bigg] dt
\end{split}\end{equation}
and
\begin{equation}\notag
\begin{split}
\widehat{\Psi} ( \w, \z) = \E \Bigg(   \exp \Bigg[ i \bigg(& \sum_{   j < J } w_j \log \bigg| \frac{c_j }{c_J} \frac{ L_j }{L_J} ( \sigma , X )\bigg| +  w_J \log |c_J L_J ( \sigma ,X ) |   \\
& + \sum_{  j < J } z_j \arg \frac{c_j}{c_J} \frac{L_j}{L_J} ( \sigma  , X ) + z_J \arg c_J L_J ( \sigma , X )  \bigg) \Bigg] \Bigg).
\end{split}\end{equation}
Now we use Lemma~\ref{lem : Dir Poly approx}  to replace the logarithms of the $L$-functions by short Dirichlet polynomials. 
Let
$$ Q_{ j,Y } (\sigma +it)= \log \frac{c_j}{c_J} +  \sum_{ p^n \leq Y } \frac{a_j(p^n) -  a_J(p^n)}{p^{n( \sigma +it)}}  $$
for $   j< J$, and 
$$ Q_{ J,Y } (\sigma +it)= \log c_J + \sum_{ p^n \leq Y } \frac{a_J(p^n)}{p^{n( \sigma +it)}}. $$
By Lemma~\ref{lem : Dir Poly approx}, for any fixed  $B_1>0$ we have 
\begin{equation}\label{D poly eq1}
 \log \frac{c_j}{c_J} \frac{L_j}{L_J } ( \sigma + it) = Q_{j,Y} ( \sigma + it) + O(( \log T)^{-B_1})\qquad (  j< J)
\end{equation}
and
\begin{equation}\label{D poly eq2}
 \log c_J L_J ( \sigma + it) =Q_{J,Y} ( \sigma + it) + O ( ( \log T)^{-B_1})
\end{equation}
for all $t \in [T,2T]$ except for a set of measure $ T^{1-d(\sigma)}$, where   $d(\sigma )>0$ is a constant. 
Here  $Y= ( \log T)^{B_2}$ and  $B_2 >2(B_1+1)/(\sigma-1/2) $. 
Letting $B_1 (T)$ be the set of $t \in [T, 2T]$ such that \eqref{D poly eq1} and \eqref{D poly eq2} hold, we then see that
\begin{equation}\notag 
\begin{split}
\widehat{\Psi}_T ( \w, \z) = & \frac1T \int_{B_1(T)} \exp \bigg[ i\bigg(\sum_{ j <  J } w_j \log \bigg| \frac{c_j }{c_J} \frac{ L_j }{L_J} ( \sigma +it)  \bigg| + w_J \log |c_J L_J ( \sigma+it) |   \\
& \hskip.8in + \sum_{ j < J } z_j \arg \frac{c_j}{c_J} \frac{L_j}{L_J} ( \sigma + it)   + z_J \arg c_J L_J ( \sigma +it) \bigg) \bigg] dt +O \big( T^{-d(\sigma)} \big) \\
 =& \frac1T \int_{B_1(T)} \exp \bigg[ i    \sum_{   j \leq J }\big( w_j \Re\ Q_{j,Y} ( \sigma +it)   + z_j \Im \   Q_{j,Y} ( \sigma + it)  \big) \bigg] dt +  O \big(( \log T)^{-B_1} \big)  \\
 =& \frac1T \int_T^{2T} \exp \bigg[ i     \sum_{   j \leq J }\Re\ \big( (  w_j - i z_j)  Q_{j,Y} ( \sigma +it)  \big) \bigg] dt +  O \big(( \log T)^{-B_1} \big) .
\end{split}
\end{equation}
  
To estimate this we define
 $$ B_2 (T) := \{ t \in [T, 2T] : |Q_{j,Y}(\s+it) | \leq (\log T)^{1-\sigma} / \log\log T ,\; j \leq J \}. $$
 Then by Lemma \ref{lem:dir poly mmt 2} and a Chebyshev inequality-type argument, we find that
 $$ \mathrm{meas} \big( [T,2T] \setminus B_2 (T) \big)   \ll T e^{-B_3 \log T/\log \log T },  $$
 where $ B_3 =  \frac{2}{3 B_2} \log \frac{ (3B_2)^{1-\sigma} }{C}$. Clearly $B_3$ will be positive if 
 we choose $B_2$  sufficiently large. Assuming this is the case, we have
 \begin{align*}
& \frac1T \int_T^{2T} \exp \bigg[ i     \sum_{   j \leq J }\Re\ \big( (  w_j - i z_j)  Q_{j,Y} ( \sigma +it)  \big) \bigg] dt  \\
 & = \frac1T \int_{B_2(T)} \exp \bigg[ i     \sum_{   j \leq J }\Re\ \big( (  w_j - i z_j)  Q_{j,Y} ( \sigma +it)  \big) \bigg] dt + O(e^{-B_3 \log T/\log \log T } ).  
 \end{align*}
Now for $|w_j|, |z_j| \leq b(\log T)^\sigma $ and $ t \in B_2 (T)$, we have
$$  \bigg| \sum_{   j \leq J }\Re\ \big( (  w_j - i z_j)  
Q_{j,Y} ( \sigma +it)  \big) \bigg|  \leq  \sqrt{2}\  b\  J \frac{ \log T}{\log \log T}.
$$
Taking  $N = [ e^2  \sqrt{2} b J  \log T/ \log \log T ] $, we see that
\be\label{exp int over B_2}
 \begin{split}
    \frac1T \int_{B_2(T)}& \exp \bigg[ i     \sum_{   j \leq J }\Re\ \big( (  w_j - i z_j)  
Q_{j,Y} ( \sigma +it)  \big) \bigg] dt  \\
 & =    \frac1T \int_{B_2(T)} \sum_{n \leq N} \frac{i^n}{n!}  \bigg( \sum_{   j \leq J }\Re\ \big( (  w_j - i z_j)  Q_{j,Y} ( \sigma +it)  \big) \bigg)^n  dt + O( e^{-N})\\
 & =    \frac1T \int_{B_2(T)} \sum_{n \leq N} \frac{i^n}{n!}  \bigg( \sum_{   j \leq J }  \big( (  w_j - i z_j)  Q_{j,Y} ( \sigma +it)  +     (  w_j + i z_j) \overline{ Q_{j,Y} ( \sigma +it) }\big) \bigg)^n  dt\\
 &\qquad  + O( e^{-N})   \\
 &=   \sum_{n \leq N} \frac{i^n}{n!} \sum_{ \mathbf{k}\cdot  \mathbf{e}+ \mathbf{k}'\cdot  \mathbf{e}=n} 
  {n \choose { \mathbf{k}, \mathbf{k'}} }
  \prod_{ j \leq J}  (  w_j - i z_j)^{k_j}    (  w_j + i z_j)^{k_j'} \\
 & \hskip1.6in \cdot  \bigg( \frac1T \int_{B_2(T)}  \prod_{j \leq J}Q_{j,Y} ( \sigma +it)^{k_j}  \overline{ Q_{j,Y} ( \sigma +it) }^{k_j'}    dt \bigg)+ O( e^{-N}),
 \end{split}
 \ee
 where $\mathbf{e} =(1, 1, \ldots,1)$,\;
$ \mathbf{k}\cdot  \mathbf{e}=k_1 + \cdots + k_J,  \;  \mathbf{k}'\cdot  \mathbf{e}=   k'_1 + \cdots + k_J' ,$
and
 $$
   {n \choose { \mathbf{k}, \mathbf{k'}} }
   = {   n \choose { k_1 , \dots , k_J , k'_1 , \dots , k_J' } }.
 $$  
We write the last integral as
\be\notag
\Bigg\{ \int_{[T, 2T]} -\int_{[T, 2T] \setminus B_2(T)} \Bigg\} \prod_{j \leq J}Q_{j,Y} ( \sigma +it)^{k_j}  \overline{ Q_{j,Y} ( \sigma +it) }^{k_j'}    dt 
\ee
By Lemma \ref{lem:dir poly mmt 2}, if  $0< b<  ( 6B_2 e^2 \sqrt{2}  J)^{-1}$, then
 \begin{align*}
    \int_{[T, 2T] \setminus B_2(T)}   \prod_{j \leq J} & Q_{j,Y} ( \sigma +it)^{k_j}  \overline{ Q_{j,Y} ( \sigma +it) }^{k_j'}    dt  \\
   & \leq \big(\mathrm{meas}\big([T, 2T] \setminus B_2(T)\big)  \big)^{1/2} 
   \bigg( \int_T^{2T}   \prod_{j \leq J}|Q_{j,Y} ( \sigma +it)|^{2(k_j+k_j')}     dt \bigg)^{1/2} \\
   & \ll T^{1/2} e^{-  B_3 \log T/(2 \log \log T) }   \prod_{j \leq J} \bigg( \int_T^{2T}  |Q_{j,Y} ( \sigma +it)|^{2n}     dt   \bigg)^{ (k_j + k_j')/(2n)} \\ 
    & \ll T  e^{-  B_3 \log T/(2 \log \log T) }   \bigg(   \frac{ Cn^{1-\sigma}  }{ (\log n)^\sigma  }  \bigg)^n.
 \end{align*}
The contribution of this integral to \eqref{exp int over B_2} is therefore 
   \begin{align*}
  & \ll    e^{-  B_3 \log T/(2 \log \log T) }    \sum_{n \leq N} \frac{1}{n!} 
\sum_{ \mathbf{k}\cdot  \mathbf{e}+ \mathbf{k}'\cdot  \mathbf{e}=n} 
  {n \choose { \mathbf{k}, \mathbf{k'}} }
 (\sqrt{2}b(\log T)^{\sigma})^n  \bigg(   \frac{ Cn^{1-\sigma}  }{ (\log n)^\sigma  }  \bigg)^n \\
  & =    e^{-  B_3 \log T/(2 \log \log T) }    \sum_{n \leq N} \frac{1}{n!}(2J)^n  (\sqrt{2}b(\log T)^{\sigma})^n  \bigg(   \frac{ Cn^{1-\sigma}  }{ (\log n)^\sigma  }  \bigg)^n  \\
    & \ll    e^{-  B_3 \log T/(2 \log \log T) }    \sum_{n \leq N} \frac{1}{n!}    \bigg( 2J \sqrt{2}bC (\log T)^{\sigma}   \frac{  N^{1-\sigma}  }{ (\log N)^\sigma  }  \bigg)^n  \\
     & \ll    e^{-  B_3 \log T/(2 \log \log T) }   \exp \bigg( 2J \sqrt{2}bC (\log T)^{\sigma} 
       \frac{  N^{1-\sigma}  }{ (\log N)^\sigma  }  \bigg)  \\
     & \ll \exp \bigg(   \bigg( -  \frac{B_3}{2}  + 2J\sqrt{2}C(e^2 \sqrt{2} J)^{1-\sigma} b^{2-\sigma} \bigg) \frac{ \log T}{\log \log T} \bigg) \\
     & \leq  \exp \bigg(     -  \frac{B_3}{4}    \frac{ \log T}{\log \log T} \bigg) ,
 \end{align*}
 provided that $b$ also satisfies $ 0 < b <   B_3^{1/(2-\sigma)} ( 8\sqrt{2} JC (e^2 \sqrt{2}J)^{1-\sigma} )^{-1/(2-\sigma)}   $. Thus,
  \begin{align*}
  \frac1T \int_T^{2T}   \exp  \bigg[ i   &   \sum_{   j \leq J }\Re\ \big( (  w_j - i z_j)   Q_{j,Y} ( \sigma +it)  \big) \bigg] dt   \\
&  =    \sum_{n \leq N} \frac{i^n}{n!} \sum_{ \mathbf{k}\cdot  \mathbf{e}+ \mathbf{k}'\cdot  \mathbf{e}=n} 
  {n \choose { \mathbf{k}, \mathbf{k'}} }
 \prod_{ j \leq J}  (  w_j - i z_j)^{k_j}    (  w_j + i z_j)^{k_j'} \\
 & \hskip.7in  \cdot \bigg(\frac1T \int_T^{2T}  
 \prod_{j \leq J}Q_{j,Y} ( \sigma +it)^{k_j}  \overline{Q_{j,Y} ( \sigma +it) }^{k_j'}    dt \bigg)\\
 & \hskip.5in+   O( e^{-N}) +O( e^{  -(B_3/4) \log T / \log \log T }). 
 \end{align*}
 Note that if we take $b>0$ sufficiently small, the first of the two $O$-terms will be the largest. Assuming this to be the case, we see by 
 Lemmas~\ref{lem:dir poly mmt 2}  and   \ref{lem:dir poly mmt 3}
 that the above is
 \begin{align*}
 =&  \sum_{n \leq N} \frac{i^n}{n!}
 \sum_{ \mathbf{k}\cdot  \mathbf{e}+ \mathbf{k}'\cdot  \mathbf{e}=n} 
   {n \choose { \mathbf{k}, \mathbf{k'}} }
    \prod_{ j \leq J}  (  w_j - i z_j)^{k_j}    (  w_j + i z_j)^{k_j'}  \\
 & \hskip.7in  \cdot   \E \bigg(     \prod_{j \leq J}Q_{j,Y} ( \sigma ,X)^{k_j}  \overline{ Q_{j,Y} ( \sigma ,X) }^{k_j'}  \bigg)+ O( e^{-N}) \\
  = &  \sum_{n \leq N} \frac{i^n}{n!}   \E \bigg( \bigg(      \sum_{   j \leq J }\Re\ \big( (  w_j - i z_j)  Q_{j,Y} ( \sigma ,X)  \big) \bigg)^n \bigg)+ O( e^{-N}) \\
  = & \E \bigg(  \exp \bigg[    i   \sum_{   j \leq J }\Re\ \big( (  w_j - i z_j)  Q_{j,Y} ( \sigma ,X)  \big) \bigg]  \bigg)+ O( e^{-N}).
 \end{align*}
We have now shown that with appropriate choices of the parameters $B_2, B_3,$ 
and $b$
\begin{equation}\label{eqn psi hat 2}
\begin{split}
\widehat{\Psi}_T ( \w, \z) =
  & \E \bigg(  \exp \bigg[    i   \sum_{   j \leq J }\Re\ \big( (  w_j - i z_j)  Q_{j,Y} ( \sigma ,X)  \big) \bigg]  \bigg)+ O( e^{-N}) ,
\end{split}
\ee
where $N = [ e^2  \sqrt{2} b J  \log T/ \log \log T ] $.

 
Now, by   direct calculation,
$$ \E \big( \big| \log c_J L_J (\sigma, X) - Q_{J, Y} (\sigma, X) \big|^2 \big) 
=  \sum_{ p^n \geq Y}  \frac{  |a_J (p^n )|^2 }{ p^{2n\sigma}}    \ll Y^{(1-2\sigma) +\epsilon} \ll (\log T)^{ (1-2\sigma) B_2  + \epsilon} . $$
From this and Chebyshev's inequality we  easily see that
$$ \log c_J L_J(\sigma, X) - Q_{J, Y} (\sigma, X) 
= O\big( ( \log T)^{-B_1}\big), $$
except for a set of $X  \in \mathbb{T}^\infty $  of measure
 $ O((\log T)^{ (1-2\sigma) B_2  + 2 B_1 + \epsilon} ).$
Similarly, 
\begin{equation}\notag
 \log \frac{c_j}{c_J} \frac{L_j}{L_J } ( \sigma , X ) = Q_{j,Y} ( \sigma , X ) 
 + O\big(( \log T)^{-B_1}\big)\qquad \quad (  j< J)
\end{equation}
holds except for a set of $X  \in \mathbb{T}^\infty $  of measure
 $ O( (\log T)^{ (1-2\sigma) B_2  + 2 B_1 + \epsilon})$.
Thus, 
\begin{equation}\begin{split}\notag
   \E \bigg(   \exp \bigg[ i    & \sum_{    j \leq J }\Re \big( (  w_j - i z_j)  Q_{j,Y} ( \sigma , X )  \big) \bigg] \bigg)  \\
= & \E \bigg(   \exp \bigg[ i     \sum_{  j < J }\Re \big( (  w_j - i z_j) \log     \frac{c_j}{c_J}  \frac{L_j}{L_J}  ( \sigma , X )   \big)   + i \Re \big( (  w_J - i z_J)  \log c_J L_J ( \sigma , X )   \big)      \bigg] \bigg) \\
&\quad  +  O\big( (\log T)^{ (1-2\sigma)B_2 + 2 B_1 + \epsilon } \big) + O\big(( \log T)^{-B_1}\big)\\
=&  \widehat{\Psi} (\w, \z) +  O( (\log T)^{ (1-2\sigma)B_2 + 2 B_1 + \epsilon } ) 
+ O\big(( \log T)^{-B_1}\big) .
\end{split}\end{equation}
Choosing first $B_1 $ and then $B_2 $ sufficiently large  as a function of $B_1$, we can ensure that for any 
given $A>1$, the last line equals
$$  \widehat{\Psi} ( \w, \z) + O\big( ( \log T)^{-A}\big) .$$
Combining this with \eqref{eqn psi hat 2}, we see that 
$$  \widehat{\Psi}_T ( \w, \z)= \widehat{\Psi} ( \w, \z) + O\big( ( \log T)^{-A}\big) .$$
This completes the proof of the Lemma~\ref{lem: transforms}.
\end{proof}

 \begin{lemma}\label{lem:upper bound hat psi y}
 There is a positive constant $C_\s$ such that 
 $$ |  \widehat{\Psi}  (y, 0 , \dots, 0 ) |      \leq     \exp\bigg( -\frac{ C_\sigma }{4} \frac{ y^{1/\sigma}}{ \log y }  \bigg)$$
 as $ y \to \infty$.
 \end{lemma}
 
 \begin{proof} 
 \begin{align*}
  \widehat{\Psi}  (y, 0 , \dots, 0 )   &  =  \E \bigg(  \exp \bigg[  i y \log \bigg| \frac{c_1 L_1}{c_J L_J} ( \sigma, X)    \bigg| \bigg]  \bigg) \\
  &  =  \E \bigg(  \exp \bigg[  i y   \bigg(   \log \bigg| \frac{c_1  }{c_J  }      \bigg|   + \Re  \sum_{p, n}  \frac{ ( a_1 (p^n)- a_J (p^n))  X(p)^n }{p^{n\sigma} } \bigg) \bigg]  \bigg)   \\
  &  =   \bigg| \frac{c_1  }{c_J  }      \bigg|^{iy} \  \prod_p   \E \bigg( \exp \bigg[  i y   \bigg(   \Re  \sum_{ n}  \frac{ ( a_1 (p^n)- a_J (p^n))  X(p)^n }{p^{n\sigma} } \bigg) \bigg]  \bigg) .
 \end{align*}
It is easy to see that for each $p$
$$   \bigg|  \E \bigg(  \exp \bigg[  i y   \bigg(   \Re  \sum_{ n}  \frac{ ( a_1 (p^n)- a_J (p^n))  X(p)^n }{p^{n\sigma} } \bigg) \bigg]  \bigg) \bigg| \leq 1 .$$
We next show that  there are a number of  $p$  for which
 $$   \bigg|  \E \bigg(  \exp \bigg[  i y   \bigg(   \Re  \sum_{ n}  \frac{ ( a_1 (p^n)- a_J (p^n))  X(p)^n }{p^{n\sigma} } \bigg) \bigg]  \bigg)  \bigg| \leq e^{-1/3} .$$
 Since $a_j (p)$ is a real number for every prime $p$ and for each $j \leq J$, we have
  for $ y \leq p^{2 \sigma }/2 $ that
  \begin{align*}
     \E \bigg(  \exp \bigg[  i y   \bigg(   \Re  \sum_{ n} 
     & \frac{ ( a_1 (p^n)- a_J (p^n))   X(p)^n }{p^{n\sigma} } \bigg) \bigg]  \bigg)\\
         &  =  \E \bigg(  \exp \bigg[  i y   \bigg(      \frac{ a_1 (p )- a_J (p )   }{p^{ \sigma} }   \Re  X(p)    \bigg) + O \bigg(  \frac{y}{ p^{2\sigma}}   \bigg) \bigg]  \bigg) \\
       &  =  \E \bigg(  \exp \bigg[  i y   \bigg(      \frac{ a_1 (p )- a_J (p )   }{p^{ \sigma} }   \Re  X(p)    \bigg)  \bigg] \bigg)  + O \bigg(  \frac{y}{ p^{2\sigma}}   \bigg)   \\
        &  = J_0     \bigg(      \frac{ a_1 (p )- a_J (p )   }{p^{ \sigma} } y       \bigg)  + O \bigg(  \frac{y}{ p^{2\sigma}}   \bigg)   ,
 \end{align*}
 where $J_0$ is the Bessel function of order 0. 
 
 Recall that $ a_j (p) = 2 \Re \chi_j ( \frak p)$ if $p$ splits and $ \frak{p} | p $. Since $ \chi_1 \neq \chi_J $ and $ \chi_1 \neq \overline{\chi_J}$, there is an ideal class $\mathcal{C}$ such that $\chi_1 (\frak{p}) \neq \chi_J ( \frak{p} ) $ and $\chi_1 (\frak{p}) \neq \overline{ \chi_J ( \frak{p} )} $ for all $ \frak{p} \in \mathcal{C}$. Since $ | \chi_j (\frak{p})|=1$, we see that
 $$ a_1 (p) - a_J (p) = 2\Re ( \chi_1 ( \frak{p} ) - \chi_J ( \frak{p} ) ) = a \neq 0 $$
 for all $ \frak{p} \in \mathcal{C} $ and $ \frak{p} | p $.
 Using the crude inequality 
 $$ | J_0 (x) |  \leq e^{-1/2} \qquad\quad (x\geq 2),$$ 
 we find that 
 $$ \bigg| J_0     \bigg(      \frac{ a_1 (p )- a_J (p )   }{p^{ \sigma} } y       \bigg)  + O \bigg(  \frac{y}{ p^{2\sigma}}   \bigg)   \bigg| \leq  e^{-1/2} + O \bigg(  \frac{y}{ p^{2\sigma}}   \bigg)   \leq e^{-1/3}, $$
 provided that $ p$, with  $ \frak{p} \in \mathcal{C} $ and $ \frak{p} | p $, satisfies the conditions
 $$ \frac{ay}{p^\sigma}   \geq 2 \qquad \hbox{and} \qquad
  \frac{ y}{p^{2 \sigma}} \leq c $$
 for some small constant $c>0$. Therefore 
  \begin{align*}
|  \widehat{\Psi}  (y, 0 , \dots, 0 ) |  &    \leq       \prod_{p \in \mathscr P }  e^{-1/3} ,
 \end{align*}
 where $\mathscr P$ is the set of all $p$ with  $ \frak{p} \in \mathcal{C} $, $ \frak{p} | p $, and
 $$ (y/c)^{1/(2\sigma)} \leq p \leq ( ay/2)^{1/\sigma} . $$
 By Lemma 2.6 of \cite{Le2}, for example, there are 
 $$ \frac{1}{h(D)} \frac{  ( ay/2)^{1/\sigma} }{ \log ( ay/2)^{1/\sigma}  }   (1+o(1)) \sim C_\sigma \frac{ y^{1/\sigma}}{ \log y }    $$
 such $p$ as $ y \to \infty$. Thus,
$$ |  \widehat{\Psi}  (y, 0 , \dots, 0 ) |      \leq     \exp\bigg( -\frac{ C_\sigma }{4} \frac{ y^{1/\sigma}}{ \log y }  \bigg)$$
 as $ y \to \infty$.
 
 \end{proof}



\subsection{Completion of the proof of Theorem~\ref{lem: discrepancy}}

We can now prove Theorem~\ref{lem: discrepancy}.   For a sufficiently large constant $A>0$  the set 
$$ \big\{ t \in [T, 2T] : \L ( \sigma + it) \notin [ -A\sL  , A\sL]^{2J}  \big\} $$
has  small measure by Lemma \ref{lem: 2kth mmt imag}. Similarly, by Lemma \ref{2kth mmt random}
$$\P \big(  \L ( \sigma , X) \notin  [ -A\sL  , A\sL]^{2J}  \big)$$
is small. Thus it is enough to consider rectangular regions $R$ contained in $ [ -A\sL  , A\sL]^{2J} $.

 Let $\eta  = b_1 ( \log T)^\sigma $, where $b_1$ is a positive constant such that $b=  2\pi b_1 $ satisfies Lemma \ref{lem: transforms}, and define
$$ G(u) = \frac{2u}{\pi} + 2(1-u) u \cot (\pi u )$$
for $ 0  \leq u \leq 1 $. By Lemma 4.1 of \cite{Ts}, the characteristic function of the interval $[\alpha, \beta]$ is 
$$ {\bf 1}_{[\alpha, \beta]} (x) = \Im \int_0^\eta G( u/\eta) e^{2 \pi i ux} f_{\alpha, \beta} (u) \frac{du}{u} 
+ O \Bigg\{  \left( \frac{ \sin \pi \eta (x-\alpha) }{ \pi \eta (x-\alpha) } \right)^2  +  \left( \frac{ \sin \pi \eta (x-\beta) }{ \pi \eta (x-\beta) } \right)^2  \Bigg\}, $$
where $f_{\alpha, \beta} (u) = ( e^{- 2 \pi i \alpha u } - e^{ - 2 \pi i \beta u } )/2 $. Thus the characteristic function of the rectangular region $R = \prod_{j \leq J} [ \alpha_j , \beta_j ] \times \prod_{ j \leq J } [ \alpha'_j, \beta'_j ] $ is
\begin{equation}\label{char fnc}
\begin{split}
 {\bf 1}_R (  \u , \v ) = & W_{\eta, R}( \u,\v)  
 +\sum_{j \leq J}  O \Bigg\{  \bigg( \frac{ \sin \pi \eta (u_j -\alpha_j) }{ \pi \eta (u_j -\alpha_j ) } \bigg)^2  
  +   \bigg( \frac{ \sin \pi \eta (u_j -\beta_j ) }{ \pi \eta (u_j -\beta_j ) }  \bigg)^2    \\
  & \hskip1.4in   +   \bigg( \frac{ \sin \pi \eta (v_j -\alpha'_j) }{ \pi \eta (v_j -\alpha'_j ) } \bigg)^2  
  +  \bigg( \frac{ \sin \pi \eta (v_j -\beta'_j ) }{ \pi \eta (v_j -\beta'_j ) } \bigg)^2  \Bigg\},
 \end{split}
 \end{equation}
 where
\be\label{W def}
 W_{\eta, R}( \u,\v)  := \prod_{j \leq J } \left( \Im \int_0^\eta G( u/\eta) e^{2 \pi i u u_j} f_{\alpha_j , \beta_j } (u) \frac{du}{u}\right) \times \prod_{j \leq J } \left( \Im \int_0^\eta G( u/\eta) e^{2 \pi i u v_j} f_{\alpha'_j , \beta'_j } (u) \frac{du}{u}\right) .
 \ee
Assuming that $R = \prod_{j \leq J} [ \alpha_j , \beta_j ] \times \prod_{ j \leq J } [ \alpha'_j, \beta'_j ] 
\subset [ -A\sL  , A\sL]^{2J} $, we see that
  $$  \Psi_T ( R ) = \frac1T \int_T^{2T} {\bf 1}_R (  \L  ( \sigma + it)  ) dt =   \frac1T \int_T^{2T} W_{\eta, R}  (  \L  ( \sigma + it)  ) dt  + \mathscr E_1  $$
and 
 $$ \Psi  ( R)  =  \E  \big( W_{\eta,R} (     \L  ( \sigma, X ) )  \big) +\mathscr E_2 ,$$
where  $\mathscr E_1$ and $\mathscr E_2$ are the error terms arising from 
the $O$-terms in \eqref{char fnc}.
 
 To estimate these error terms  we begin with  the  identity
 $$ \frac{ \sin^2 ( \pi \eta x ) }{ ( \pi \eta x )^2 }   = \Re \bigg[   \frac{1}{2 \pi^2 \eta^2 } \int_0^{2 \pi \eta}  ( 2 \pi \eta- y) e^{  i xy } dy    \bigg] .$$
A typical term in $\mathscr E_1$ is 
 \begin{align*}
  \frac{1}{T} \int_T^{2T} &  \frac{ \sin^2 ( \pi \eta (\log | c_1 L_1/c_J L_J (s)|  - \alpha_1)) }{ ( \pi \eta  (\log | c_1 L_1/c_J L_J (s)| -\alpha_1) )^2 }   dt   \\
 &\hskip.7in = \Re \bigg(  \frac{1}{2 \pi^2 \eta^2 } \int_0^{2 \pi \eta}  ( 2 \pi \eta- y)   \frac{1}{T} \int_T^{2T}  e^{  i y (\log | c_1 L_1/c_J L_J (s)|  - \alpha_1)    } dt dy    \bigg)  \\
  & \hskip.7in = \Re \bigg(  \frac{1}{2 \pi^2 \eta^2 } \int_0^{2 \pi \eta}  ( 2 \pi \eta- y)    e^{-  i y \alpha_1} \widehat{\Psi}_T (  y, 0 , \dots, 0 )    dy    \bigg)  .
 \end{align*}
 By Lemma \ref{lem: transforms} this equals
  \begin{align*}
       \Re \bigg(   \frac{1}{2 \pi^2 \eta^2 } \int_0^{2 \pi \eta}  ( 2 \pi \eta- y)       e^{-  i y \alpha_1}   \widehat{\Psi}  (  y, 0 , \dots, 0 )    dy    \bigg)  +O(( \log T)^{-A}).
 \end{align*}
 Note that this is also a typical term in $\mathscr E_2$. By 
 Lemma~\ref{lem:upper bound hat psi y} this is
   \begin{align*}
   & \ll    \frac{1}{  \eta^2 } \int_0^{2  }  ( 2 \pi \eta- y)       dy    +  \frac{1}{  \eta^2 } \int_2^{2 \pi \eta  }  ( 2 \pi \eta- y)       \exp\bigg( -\frac{ C_\sigma }{4} \frac{ y^{1/\sigma}}{ \log y }  \bigg) dy  +O(( \log T)^{-A})\\
   &\ll \frac{1}{\eta}.
 \end{align*}
 All the other terms  in $\mathscr E_1$ and $\mathscr E_2$ are estimated similarly.
 Thus, it is enough to show that
 $$ \frac1T \int_T^{2T} W_{\eta, R}  (  \L  ( \sigma + it)  ) dt =  \E  \big( W_{\eta,R} (     \L  ( \sigma, X ) )  \big) + O( 1/\eta )   .$$

Using $ \Im z = (z - \bar{z} ) / 2i $ to rewrite the imaginary parts in \eqref{W def},  we see that
\begin{equation}\notag
\begin{split}
 W_{\eta, R}( \u,\v)  =& (-4)^{-J} \sum_{ \epsilon_j , \epsilon'_j = \pm 1}   
 \Bigg\{ \prod_{j \leq J } \bigg(   \int_0^\eta G( w_j /\eta) \epsilon_j e^{\epsilon_j 2 \pi i w_j u_j} f_{\epsilon_j\alpha_j , \epsilon_j\beta_j } (w_j) \frac{dw_j}{w_j }\bigg)  \\ 
 &\hskip1.2in  \times \prod_{j \leq J } \bigg(   \int_0^\eta G( z_j/\eta) \epsilon'_j e^{ \epsilon'_j 2 \pi i z_j v_j} f_{ \epsilon'_j \alpha'_j ,  \epsilon'_j \beta'_j } (z_j) \frac{dz_j}{z_j}\bigg)  \Bigg\} \\
 =& (-4)^{-J} \sum_{ \epsilon_j , \epsilon'_j = \pm 1} \int_{[0,\eta]^{2J}} \prod_{j \leq J} 
 \Big\{G( w_j /\eta)  G( z_j/\eta)  \epsilon_j  \epsilon'_j  f_{\epsilon_j\alpha_j , \epsilon_j\beta_j } (w_j)  
   f_{ \epsilon'_j \alpha'_j ,  \epsilon'_j \beta'_j } (z_j) \Big\}  \\ 
& \hskip1.2in \times   \exp \bigg( { \sum_{j \leq J} \epsilon_j 2 \pi i w_j u_j + \sum_{j \leq J}   \epsilon'_j 2 \pi i z_j v_j} \bigg)   \frac{dw_1}{w_1} \cdots \frac{dw_J}{w_J} \frac{dz_1 }{z_1} \cdots \frac{dz_J}{z_J}. 
 \end{split}\end{equation}
Thus
 \begin{equation}\notag
 \begin{split}
 \frac1T \int_T^{2T} W_{\eta, R} & (  \L  ( \sigma + it)  ) dt \\
 =  &  (-4)^{-J} \sum_{ \epsilon_j , \epsilon'_j = \pm 1} \int_{[0,\eta]^{2J}} \prod_{j \leq J} \big(G( w_j /\eta)  G( z_j/\eta)  \epsilon_j  \epsilon'_j  f_{\epsilon_j\alpha_j , \epsilon_j\beta_j } (w_j)  
   f_{ \epsilon'_j \alpha'_j ,  \epsilon'_j \beta'_j } (z_j) \big)  \\ 
&\quad\qquad  \widehat{\Psi}_T ( 2 \pi  \epsilon_1 w_1 , \dots, 2 \pi \epsilon_J w_J, 2 \pi \epsilon'_1 z_1 , \dots,  2 \pi  \epsilon'_J z_J  )   \frac{dw_1}{w_1} \cdots \frac{dw_J}{w_J} \frac{dz_1 }{z_1} \cdots \frac{dz_J}{z_J} 
 \end{split}\end{equation}
and
\begin{equation}\notag
\begin{split}
   \E  \big( W_{\eta,R} &(     \L  ( \sigma, X ) )  \big)\\
     =  &   (-4)^{-J} \sum_{ \epsilon_j , \epsilon'_j = \pm 1} \int_{[0,\eta]^{2J}} \prod_{j \leq J} \big(G( w_j /\eta)  G( z_j/\eta)  \epsilon_j  \epsilon'_j  f_{\epsilon_j\alpha_j , \epsilon_j\beta_j } (w_j)  
   f_{ \epsilon'_j \alpha'_j ,  \epsilon'_j \beta'_j } (z_j) \big)  \\ 
& \quad\qquad   \widehat{\Psi} ( 2 \pi  \epsilon_1 w_1 , \dots, 2 \pi \epsilon_J w_J, 2 \pi \epsilon'_1 z_1 , \dots,  2 \pi  \epsilon'_J z_J  )   \frac{dw_1}{w_1} \cdots \frac{dw_J}{w_J} \frac{dz_1 }{z_1} \cdots \frac{dz_J}{z_J} . 
 \end{split}\end{equation}
 Since
$$ \widehat{\Psi}_T ( \w, \z) - \widehat{\Psi} (\w,\z) = O ( ( \log T)^{-A})$$
  for all $ |w_j|, |z_j| \leq \eta $ by Lemma~\ref{lem: transforms}, we have
 \begin{equation}\begin{split}
 & \frac1T \int_T^{2T} W_{\eta, R}  (  \L  ( \sigma + it)  ) dt  -    \E  \big( W_{\eta,R} (     \L  ( \sigma, X ) )  \big)  \\
 \ll  &   \sum_{ \epsilon_j , \epsilon'_j = \pm 1} \int_{[0,\eta]^{2J}} \prod_{j \leq J} \big| G( w_j /\eta)  G( z_j/\eta)     f_{\epsilon_j\alpha_j , \epsilon_j\beta_j } (w_j)    f_{ \epsilon'_j \alpha'_j ,  \epsilon'_j \beta'_j } (z_j) \big|   \frac{1}{( \log T)^{A}}   \frac{dw_1}{w_1} \cdots \frac{dw_J}{w_J} \frac{dz_1 }{z_1} \cdots \frac{dz_J}{z_J}  \\
 \ll &   \frac{ \eta^{2J} \sL^{2J} }{ (\log T)^{A}} \\
 \ll & 1/\eta,
 \end{split}
 \end{equation}
provided we choose $A >0 $ sufficiently large. Here, we have used the inequalities
$$ 0 \leq G(x) \leq 2/\pi$$
for $ x \in [0,1]$ and 
$$f_{\alpha, \beta}(u) \ll |(\beta-\alpha)u|\ll \sL |u|.$$
This completes the proof of Theorem~\ref{lem: discrepancy}.


\section{Proof of Theorem~\ref{main thrm} }\label{begin proof}

Let $J \geq 2 $ and for each $j \leq J$ define
$$S_j (T) = \{ t \in [T,2T] : |c_j L_j  ( \sigma + it) 
| \geq |c_i L_i ( \sigma +it) | \mathrm{~for~ all~} i \leq J \} .$$
 By the inclusion-exclusion principle we see that
\begin{align}\label{log E est 1}\notag
 \frac1T \int_T^{2T} \log |E (\sigma+it, Q)| dt  
 = &  \sum_{1\leq j \leq J} \frac1T \int_{S_j(T)} \log |E(\sigma+it, Q)| dt \\
 & - \sum_{1\leq   j_1 < j_2 \leq J } \frac1T 
 \int_{S_{j_1} (T)\, \cap\, S_{j_2} (T)}  \log |E (\sigma + it, Q) | dt    \\
 &+\sum_{1\leq   j_1 < j_2<j_3 \leq J } \frac1T 
 \int_{S_{j_1} (T)\, \cap \,S_{j_2} (T)\, \cap\, S_{j_3} (T)}  \log |E (\sigma + it, Q) | dt   \notag    \\
 &+ \cdots 
     \notag \\
 & +(-1)^{J-1}  \frac1T \int_{S_{1} (T) \,\cap \,S_{2} (T)\, \cap    \,
 \cdots \,   \cap\,  S_{J} (T)}  \log |E (\sigma + it, Q) | dt .\notag
 \end{align}
We first show that  for $1\leq i\neq j \leq J$  we have
$$
 \int_{S_{i} (T)\, \cap\, S_{j } (T)}  |\, \log |E (\sigma + it, Q) | \, |dt
\ll T\frac{\sL^2}{(\log T)^\sigma} .
$$
 Without loss of generality, we may assume that   $i=1$ and $j=J$. 
 By H\"older's inequality and Lemma~\ref{lem: 2kth mmt real}
\begin{align}\label{S1 cap S2} \notag
 \int_{S_{1} (T)\, \cap\, S_{J } (T)   } & |\, \log |E (\sigma + it, Q) |\,| \, dt \\
 &\ll  \left(   \int_T^{2T} | \, \log |E (\sigma + it, Q) |\, |^{2k} dt \right)^{\frac{1}{2k}} 
\times     \mathrm{meas}\,  (  S_{1} (T)\, \cap \, S_{J} (T) )^{ 1- \frac{1}{2k} }     \\
&\ll   k^2 \  T^{\frac1{2k}}\, \big(\mathrm{meas}\,  (S_{1} (T)\, \cap \, S_{J} (T) )\big)^{1- \frac{1}{2k} } .\notag
\end{align}
 Now
\begin{align*} 
 S_{1} (T) \cap  S_{J} (T)
\subseteq & \
\{ \;  t \in [T,2T] : |c_1L_1  ( s) |=
|c_J L_J (s) | \; \}  \\
 = &\
\{ \;  t \in [T,2T] :    \log   | c_1 L_1  (s)/c_J L_J (s) |   =0  \} \\
= & \{ \;  t \in [T,2T] : \L(s) \in   \{0\}   
 \times   {\mathbb R }^{2J-1}   \} .
\end{align*}
By Lemma~\ref{lem: discrepancy},  and since $\Psi$ is absolutely continuous, 
the measure of this set equals
 \be\notag
\begin{split}
 T \  &\Psi  \big(  \{0\}   
 \times   {\mathbb R }^{2J-1}   \big)
+O(T(\log T)^{-\sigma})  =O(T(\log T)^{-\sigma}).
\end{split}
\ee
Hence
\be\notag
\mathrm{meas} \  (S_1 (T) \, \cap \, S_J (T) )
\ll  T(\log T)^{-\sigma} .
\ee
Thus, by  \eqref{S1 cap S2} we see that 
$$
 \int_{S_{1} (T)\, \cap\, S_{J } (T)   }  |\, \log |E (\sigma + it, Q) |\,| \, dt 
 \ll 
 k^2 T \, \bigg(\frac{1}{  (\log T)^{ \sigma} } \bigg)^{1-\frac{1}{2k}}.
$$
Taking  $k=\sL$, we obtain
$$
 \int_{S_{1} (T)\, \cap\, S_{J } (T)   }  |  \log |E (\sigma + it, Q) |\ | \, dt 
 \ll 
   T \, \frac{\sL^2}{  (\log T)^{ \sigma} } ,
$$
as claimed.

Clearly this bound applies to all the integrals   on the right-hand side of  
\eqref{log E est 1} that involve 
 two or more $S_i(T)$'s. Thus we see that
\begin{align}\notag\label{eqn: E int 1} 
 \frac1T \int_T^{2T} \log |E (\sigma+it, Q)| dt  
 = &  \sum_{j \leq J} \frac1T \int_{S_j(T)} \log |E(\sigma+it, Q)| dt 
 +O  \bigg(  \frac{2^J  \sL^2}{ (\log T)^{ \s} } \bigg).
  \end{align}
Our next task is to estimate the individual terms here and, without loss of generality, 
we consider only the case $j=J$.  
We write
\be\label{eqn: E int split}
\begin{split} 
\frac1T & \int_{S_J(T)} \log  |E(\sigma+it, Q)| dt \\
=  &  \frac1T \,\int_{S_J (T)} \log |c_J L_J (\sigma+it) | dt 
+  \frac1T \int_{S_J (T)} 
\log \bigg|1 + \sum_{  j <J }  \frac{c_j}{c_J}  \frac{L_j }{ L_J} (\sigma+it) \bigg| dt ,  
\end{split}
\ee
and calculate the integrals on the right in the next two subsections.

 
\subsection{The first integral in \eqref{eqn: E int split} }\label{first integral}

Recall that
$$
S_J (T)   = \big\{ t \in [T,2T] : |c_J L_J ( \sigma + it) 
| \geq |c_j L_j ( \sigma +it) | ,\;  j < J \big\} .
$$
 and 
 $$I_T =(-A\sL , A\sL ].$$
 We also define
 $$ I_T^{-} =(-A \sL , 0] .$$ 
 By Lemma~\ref{lem: 2kth mmt real} we may restrict the range of integration to the set
\be\notag
 \begin{split}
  S_{J,1} (T)  = \bigg\{ t \in  S_J(T) : &  \log |c_J L_J ( \sigma + it)| \in I_T, \;
\arg c_J L_J ( \sigma+it) \in I_T, 
  \\
  &  \log \bigg|\frac{c_jL_j }{c_J L_J } ( \sigma + it) \bigg|\in  I_T^{-},\;   \arg c_jL_j ( \sigma + it) - \arg c_J L_J ( \sigma + it) \in I_T   ,  j < J    \bigg\}  
  \end{split}
\ee
 at the cost of an error term of size $O(T(\log T)^{-B})$ . That is, 
\be\label{S1 int} 
\frac1T \int_{S_J (T)} \log |c_J L_J ( \sigma+it) | dt 
=  \frac1T \int_{S_{J, 1}(T) } \log |c_J L_J (\sigma+it ) | dt + O( ( \log T)^{-B}). 
\ee

  Letting   $ \u = ( u_1 , \dots , u_J ) $ and $\v = (v_1, \dots, v_J )$, we see that
\begin{equation}  \label{eqn: 1}
\begin{split}
 &  \frac1T \int_{S_{J,1} (T)} \log |c_J L_J (\sigma+it) | \; dt   \\
 &  =   \int_{    (I_T^{-})^{J-1} \times I_T \times {(I_T)}^J  }     u_J    \;  d\Psi_T (\u, \v) \\
 &  =   \int_{  (I_T^{-})^{J-1} \times I_T \times (I_T)^J}    
 \bigg( \int_{-A\sL }^{u_J} du - A\sL  \bigg) \, d\Psi_T  (\u ,\v) \\
  &  =   \int_{ (I_T^{-})^{J-1} \times I_T \times (I_T)^J}     
  \bigg( \int_{- A\sL  }^{u_J} du\; \bigg)  d\Psi_T (\u, \v )   
   - A\sL  \; \Psi_T \big( (I_T^{-})^{J-1} \times I_T \times {(I_T)}^J \big) \\
    &  =   \int_{ I_T } \bigg(  \int_{ (I_T^{-})^{J-1} \times  (u , A\sL ]  \times {(I_T)}^J  }    
    d\Psi_T  (\u , \v)  \bigg) d {u}   
    -  A\sL \;  \Psi_T \big( (I_T^{-})^{J-1} \times I_T \times {(I_T)}^J  \big) \\
        &  =   \int_{I_T}   \Psi_T\big(   (I_T^{-})^{J-1} \times  (u , A\sL ]  \times {(I_T)}^J  \big) 
              d{u}   -A\sL \; \Psi_T 
               \big( (I_T^{-})^{J-1} \times I_T \times {(I_T)}^J \big)  .
 \end{split}
 \end{equation}  
Thus, by   Lemma~\ref{lem: discrepancy}, the last line equals
\begin{align*} 
 \int_{I_T }   \Psi \big(  (I_T^{-})^{J-1} \times  (u , A\sL ]  \times {(I_T)}^J  \big) 
d{u} &  -A\sL \; \Psi \big( (I_T^{-})^{J-1} \times I_T \times {(I_T)}^J  \big) 
 +O \bigg( \frac{ \sL }{ (\log T)^\sigma} \bigg) .
\end{align*}
Now simply reverse all the steps leading to the last line of \eqref{eqn: 1} to see that
\be\notag
 \frac1T \int_{S_{J, 1} (T)} \log |c_J L_J (\sigma+it) | \; dt   
 =  \int_{    (I_T^{-})^{J-1} \times I_T \times {(I_T)}^J  }     u_J    \;  d\Psi (\u, \v)  
 +O \bigg( \frac{ \sL }{ (\log T)^\sigma} \bigg) .
\ee
%
 Therefore, by \eqref{S1 int}
 \be\label{S1 int 2} 
 \frac1T \int_{S_{J} (T)} \log |c_J L_J (\sigma+it) | \; dt   
 =  \int_{    (I_T^{-})^{J-1} \times I_T \times {(I_T)}^J  }     u_J    \;  d\Psi (\u, \v)  
 +O \bigg( \frac{ \sL }{ (\log T)^\sigma} \bigg) .
\ee


\subsection{The second integral in \eqref{eqn: E int split}}\label{step 2}

As with the first integral in \eqref{eqn: E int split}, we begin by limiting the range of integration. 
Let
\begin{equation}\notag
\begin{split}
S_{J,2} (T)  := &  \big\{ t \in [T, 2T]):    \log|c_J L_J ( \sigma +it) | \in I_T, \arg c_J L_J ( \sigma +it) \in I_T, \\
  & \log |c_jL_j / c_JL_J  ( \sigma + it) |  \in I_T^{-} , 
    \arg c_jL_j( \sigma +it)-c_J L_J ( \sigma +it) \in I_T  \mathrm{~for~ all ~}  j < J \big\}  \\
   \subseteq & S_J (T)= \big\{ t \in [T,2T] : |c_J L_J ( \sigma + it) 
| \geq |c_j L_j ( \sigma +it) | ,\;  j < J \big\} .
\end{split}
\end{equation}
Then
$$  \frac1T \int_{S_J (T)} \log \bigg|1 + \sum_{  j <  J }  
\frac{c_j}{c_J}  \frac{L_j }{ L_J} (\sigma+it) \bigg| dt  
=  \frac1T \int_{\S_{J,2} (T)} \log \bigg|1 
+ \sum_{   j < J }  \frac{c_j}{c_J}  
\frac{L_j }{ L_J} (\sigma+it) \bigg| dt  + O(( \log T)^{-B})
$$
by Lemmas~\ref{lem: 2kth mmt real} and ~\ref{lem: 2kth mmt imag}.

 Let $\u = ( u_1 , \dots , u_J) $ and $ \v = ( v_1 , \dots , v_J)$ be two vectors in $\mathbb{R}^J$. We would like to show that
$$  
\frac1T \int_{S_{J,2} (T)} \log \bigg|1 + \sum_{  j < J }  \frac{c_j}{c_J}  \frac{L_j }{ L_J} (\s+it) \bigg| dt =  \int_{ (I^{-}_T)^{J-1}\times  I_T  \times (I_T)^{J}   }   \log \bigg| 1 + \sum_{ j \leq J -1} e^{u_j + i v_j } \bigg|         d\Psi_{T} ( \u, \v)  ,
$$
but this is not straightforward because  the integrand has  logarithmic singularities. To handle this we split the integral into small pieces by dividing the set $ (I^{-}_T)^{J-1}\times I_T \times (I_T)^{J}  $ into at most $O(\sL)$ rectangular regions.

Let $\delta   = ( \log T)^{- c_0}$ with a small constant $c_0>0$. Let $\m_{J-2} = ( m_1 , \dots, m_{J-2})$ and $ \n_{J-2} = ( n_1, \dots , n_{J-2})$ be two vectors in $\mathbb{Z}^{J-2}$ and $ \u_{J-2} = ( u_1, \dots , u_{J-2} ) $ and $\v_{J-2} = ( v_1 , \dots , v_{J-2}) $ be projections of 
the vectors $\u = ( u_1 , \dots , u_{J}) $ and $ \v = ( v_1 , \dots , v_{J})$ into $\mathbb{R}^{J-2}$. Define a $(2J-4)$-dimensional rectangular region 
\begin{align*}
  {Rect}(\m,\n) := \big\{ (\u_{J-2} , \v_{J-2} )  \in (I_T^{-})^{J-2}  \times (I_T &)^{J-2}  :  \;
      m_j  \delta < e^{u_j } \leq (m_j +1)\delta , \\
      &  n_j \delta < v_j \leq  (n_j +1 ) \delta \  \mathrm{~for~all~} j \leq J-2\ \big\}         
\end{align*}
for $ 0 \leq m_j \leq 1/\delta -1 $, $ | n_j | \leq A\sL / \delta$, and $ j \leq J-2$.
Then 
$$ (I^{-}_T)^{J-2} \times (I_T)^{J-2} = \bigcup_{\m, \n} Rect(\m,\n).$$
If $m< 1/ \delta  $, the set 
$$ \big\{ e^{u+iv} \in \mathbb{C} \; : \;  m\delta <  e^u \leq (m+1) \delta ,\; n \delta < v \leq (n+1) \delta \    \big\} $$
 has  diameter    $ \leq \delta $. The set
$$\bigg\{ 1 + \sum_{j \leq J-2} e^{u_j + i v_j } \in \mathbb{C} \  :  \  (\u_{J-2} , \v_{J-2}) \in Rect(\m,\n )     \bigg\}  $$ 
is therefore   contained in a circle in $\mathbb{C}$ of radius at most $J\delta $. Let  $s_0$ be the center of this circle. Then since each $u_j\leq 0$, we have   $|s_0| \leq J-1$. We consider four cases depending on the size of $|s_0|$.

\noindent{\bf Case 1:   $|s_0 | \leq 10 J \delta $.}\,
\ 
Define
   $$ R_{main} (\m,\n) = \{ (\u,\v) \in  (I^{-}_T)^{J-1}\times I_T \times (I_T)^{J}  :\; (\u_{J-2}, \v_{J-2} ) \in Rect(\m,\n) , \ 12J\delta  <e^ {u_{J-1}} \leq 1      \} $$
   and 
\begin{align*} 
R_{error} (\m,\n) = &\{ (\u,\v) \in  (I^{-}_T)^{J-1}\times I_T \times (I_T)^{J}  :\; (\u_{J-2}, \v_{J-2} ) \in Rect(\m,\n) ,  \
e^{-A\sL} < e^ {u_{J-1}} \leq 12J\delta  \} ,
\end{align*}
so that
$$ (I^{-}_T)^{J-1}\times I_T \times (I_T)^{J}= \bigcup_{\m,\n} \Big(R_{main} (\m,\n) \bigcup R_{error} (\m,\n) \Big) . $$
Then
\be\label{sum lower bd 1} 
\bigg|  1 + \sum_{j \leq J-1} e^{u_j + i v_j }  \bigg| \geq e^{u_{J-1}} - \bigg|  1 + \sum_{j \leq J-2} e^{u_j + i v_j }  \bigg|  \geq  J \delta 
\ee
for $(\u,\v) \in R_{main}(\m,\n)$, and
$$\bigg|  1 + \sum_{j \leq J-1} e^{u_j + i v_j }  \bigg| 
\leq \bigg|  1 + \sum_{j \leq J-2} e^{u_j + i v_j }  \bigg| + e^{u_{J-1}} \leq 23J\delta $$
for $(\u,\v) \in R_{error}(\m,\n)$ .
 Observe that 
 $$
\log \bigg| 1 + \sum_{ j \leq J -1} e^{u_j + i v_j } \bigg|      
$$
may have singularities on $R_{error}(\m,\n)$, but not on $R_{main}(\m,\n)$
because of \eqref{sum lower bd 1}.

\noindent{\bf  Case 2:  $10 J \delta < |s_0| \leq 1-2J\delta$.} \, 
In this case the inequality 
  $$ \bigg|  1 + \sum_{j \leq J-1} e^{u_j + i v_j }  \bigg| \geq  J \delta $$
  holds if $ | e^{u_{J-1 } + i v_{J-1} } + s_0 | \geq 2J\delta $. We define
  $$ R_{main} (\m,\n) =  R_1 (\m,\n) \bigcup  R_2 (\m,\n) \bigcup 
 \bigg( \bigcup_{ \substack{l\in \mathbb Z\\ |l|  \leq A\sL/2\pi}}  R_{3, l} (\m,\n)  \bigg),
  $$
  where
  \be\notag
 R_1 (\m,\n)  =  \{  (\u,\v) \in  (I^{-}_T)^{J-1}\times I_T \times (I_T)^{J}  : (\u_{J-2} , \v_{J-2} ) \in Rect (\m,\n) ,  
   e^{-A\sL } < e^ {u_{J-1}} \leq |s_0|-2J\delta   \} , 
   \ee
   \be\notag
R_2 (\m,\n) = \{   (\u,\v) \in  (I^{-}_T)^{J-1}\times I_T \times (I_T)^{J} : (\u_{J-2} , \v_{J-2} ) \in Rect (\m,\n) ,   |s_0| + 2J\delta < e^{u_{J-1}} \leq 1    \}  ,
\ee
    and
 \begin{align*} 
R_{3 , l}(\m,\n) =  
 \{  (\u,\v) \in &  (I^{-}_T)^{J-1}\times I_T \times (I_T)^{J}  : (\u_{J-2} , \v_{J-2} ) \in Rect (\m,\n) ,   
  |s_0| -2J\delta<   e^ {u_{J-1}}  \leq  |s_0|+  2J\delta ,  \\
 & -(\pi-\arcsin ( 2J\delta/|s_0|) )< v_{J-1} - \arg s_0 -2\pi l   \leq \pi-\arcsin ( 2J\delta/|s_0|)  \}.
 \end{align*}
Also define 
      \begin{equation*}\begin{split}
   R_{error} (\m,\n) =&\bigcup_{ \substack{l\in \mathbb Z\\ |l|  \leq A\sL/2\pi}} 
    \{ (\u,\v) \in  (I^{-}_T)^{J-1}\times I_T \times (I_T)^{J}  : (\u_{J-2} , \v_{J-2} ) \in Rect (\m,\n) , \\
   & \hskip.9in      |s_0| - 2J\delta<      e^ {u_{J-1}}     \leq  |s_0|+ 2J\delta ,  \\
   & \hskip.9in   -\arcsin ( 2J\delta/|s_0|)<v_{J-1} - \arg s_0  -(2l+1)\pi  \leq \arcsin ( 2J\delta/|s_0|)   \}.
   \end{split}\end{equation*}
Then we see that the sets  $R_{main} (\m,\n) $ and  $R_{error} (\m,\n)$  are each  unions of roughly 
$ A\sL / \pi $   rectangular regions, and that 
\be\notag\label{sum lower bd 2} 
\bigg|  1 + \sum_{j \leq J-1} e^{u_j + i v_j }  \bigg| \geq  J \delta 
\ee
   for $(\u,\v) \in R_{main}(\m,\n)$, and
$$ \bigg|  1 + \sum_{j \leq J-1} e^{u_j + i v_j }   \bigg| 
\leq  \bigg|  1 + \sum_{j \leq J-2} e^{u_j + i v_j } -s_0  \bigg| 
+  \bigg| s_0 + e^{u_{J-1}+iv_{J-1}}  \bigg| \leq 10J\delta $$
    for $(\u,\v) \in R_{error}(\m,\n)$.

\noindent{\bf  Case 3:  $ 1-2J\delta  < |s_0| \leq 1+2J\delta$.} \, 
Similarly to Case 2, we define
  $$ R_{main} (\m,\n) =  R_1 (\m,\n)   \bigcup 
 \bigg( \bigcup_{ \substack{l\in \mathbb Z\\ |l|  \leq A\sL/2\pi}}  R_{4, l} (\m,\n)  \bigg),
  $$
  where
 \begin{align*} 
R_{4 , l}(\m,\n) =  
 \{ (\u,\v) \in & (I^{-}_T)^{J-1}\times I_T \times (I_T)^{J}  : (\u_{J-2} , \v_{J-2} ) \in Rect (\m,\n) ,   
   |s_0| -2J\delta<   e^ {u_{J-1}}  \leq 1 ,  \\
 & -(\pi-\arcsin ( 2J\delta/|s_0|) )< v_{J-1} - \arg s_0 -2\pi l   \leq \pi-\arcsin ( 2J\delta/|s_0|)  \}.
 \end{align*}
Also define 
      \begin{equation*}\begin{split}
   R_{error} (\m,\n) =&\bigcup_{ \substack{l\in \mathbb Z\\ |l|  \leq A\sL/2\pi}} 
    \{  (\u,\v) \in (I^{-}_T)^{J-1}\times I_T \times (I_T)^{J}   : (\u_{J-2} , \v_{J-2} ) \in Rect (\m,\n) ,  \\
   & \hskip.9in    
 |s_0|  - 2J\delta<      e^ {u_{J-1}}       \leq 1 ,  \\
   & \hskip.9in   -\arcsin ( 2J\delta/|s_0|)<v_{J-1} - \arg s_0  -(2l+1)\pi  \leq \arcsin ( 2J\delta/|s_0|)   \}.
   \end{split}\end{equation*}
Then we see that the sets  $R_{main} (\m,\n) $ and  $R_{error} (\m,\n)$  are each  unions of roughly 
$ A\sL / \pi $   rectangular regions, and that 
\be\notag
\bigg|  1 + \sum_{j \leq J-1} e^{u_j + i v_j }  \bigg| \geq  J \delta 
\ee
   for $(\u,\v) \in R_{main}(\m,\n)$, and
$$ \bigg|  1 + \sum_{j \leq J-1} e^{u_j + i v_j }   \bigg| 
\leq  \bigg|  1 + \sum_{j \leq J-2} e^{u_j + i v_j } -s_0  \bigg| 
+  \bigg| s_0 + e^{u_{J-1}+iv_{J-1}}  \bigg| \leq 10J\delta $$
    for $(\u,\v) \in R_{error}(\m,\n)$.

\noindent{\bf  Case 4:  $ 1+2J\delta < |s_0| \leq  J -1 $.} \, 
We define
$$ R_{main} (\m,\n) =  \{   (\u,\v) \in  (I^{-}_T)^{J-1}\times I_T \times (I_T)^{J}  : (\u_{J-2} , \v_{J-2} ) \in Rect (\m,\n)   \} 
  $$
and  
$$ R_{error} (\m, \n) = \emptyset .$$
Then
 \be\notag
\bigg|  1 + \sum_{j \leq J-1} e^{u_j + i v_j }  \bigg|  \geq \bigg|  1 + \sum_{j \leq J-1} e^{u_j + i v_j }  \bigg|  - e^{u_{J-1}} \geq (|s_0| - J\delta)  - 1 \geq  J \delta 
\ee
   for $(\u,\v) \in R_{main}(\m,\n)$.



Summarizing, we note  that in each case, we have
\be\label{lower bd 3}
\bigg|  1 + \sum_{j \leq J-1} e^{u_j + i v_j }  \bigg| \geq  J \delta 
\ee
if $(\u,\v) \in R_{main}(\m,\n)$,  
\be\label{upper bd 3}
\bigg|  1 + \sum_{j \leq J-1} e^{u_j + i v_j }  \bigg| \leq  23 J \delta 
\ee
if $(\u,\v) \in R_{error}(\m,\n)$, and
\be\label{set decomp}
  (I^{-}_T)^{J-1}\times I_T \times (I_T)^{J} = \bigcup_{\m,\n} \Big(R_{main} (\m,\n) \bigcup R_{error} (\m,\n) \Big) . 
\ee
We also note that  for each $u_j $ with $ j <J$ we have
\be\label{uj bound}
e^{u_j} \leq 1,
\ee
since $ \u \in  (I^{-}_T)^{J-1}\times I_T $ and $ \v \in  (I_T)^{J}$.

We now write 
\begin{equation}\label{eqn:decompose}\begin{split}
   \frac1T \int_{S_{J,2} (T)} \log \bigg|1 + \sum_{  j < J }  
   \frac{c_j}{c_J}  \frac{L_j }{ L_J} (\sigma+it)  \bigg| dt   
 = & 
 \frac1T \int_{S_{main} (T)} \log  \bigg|1 
 + \sum_{  j < J }  \frac{c_j}{c_J}  \frac{L_j }{ L_J} (\sigma+it)  \bigg| dt  \\
  &+ \frac1T \int_{S_{error} (T)} \log  \bigg|1 
  + \sum_{  j < J }  \frac{c_j}{c_J}  \frac{L_j }{ L_J} (\sigma+it)  \bigg| dt ,
\end{split}\end{equation}
 where
  $$ S_{ main}(T) = \bigcup_{\m,\n} S_{ main} ( \m, \n , T) ,$$
 $$ S_{error}(T) = \bigcup_{\m,\n} S_{error} ( \m, \n , T),$$
 and
$$ S_{main}(\m,\n,T) = \{ t \in [ T, 2T] : \L ( \sigma+it) \in R_{main} (\m,\n) \}, $$
$$ S_{ error}(\m,\n,T) = \{ t \in [ T, 2T] : \L ( \sigma+it) \in R_{error} (\m,\n) \} .$$
  
Here, the $(\m,\n)$-sum and the $(\m,\n)$-union   are over $ 0 \leq m_j \leq 1/\delta -1 $ and $ | n_j | \leq A \sL / \delta$ for $ j \leq J-2$.

\noindent{\bf The main term:} \
For each $\m, \n$, we have
$$  \frac1T \int_{S_{ main} (\m,\n,T)} \log  \bigg|1 + \sum_{   j< J }  \frac{c_j}{c_J}  \frac{L_j }{ L_J} (\sigma+it)  \bigg| dt =  \int_{ R_{main}(\m,\n) }   \log \bigg| 1 + \sum_{ j < J } e^{u_j + i v_j } \bigg|      \;   d\Psi_{ T} ( \u, \v)  ,$$
and we wish to express this in terms of the distribution function $\Psi ( \u, \v) $.
Since each $R_{main}(\m,\n)$ is a union of rectangular regions, we let 
$$R= \prod_{ j \leq 2J} ( a_j, b_j ] $$ 
be one of them and consider
$$ \int_{ R  }   \log \bigg| 1 + \sum_{ j <J} e^{u_j + i v_j } \bigg|    \;     d\Psi_{T} ( \u, \v) .  $$
Our argument is similar to that in   Subsection~\ref{first integral}, but slightly more complicated.
For $\w = ( w_1, \dots, w_{2J})$ define
\begin{equation}\label{h formulas}
\begin{split}
h_0 ( \w )  = & \log \bigg| 1 + \sum_{ j <J} e^{w_j + i w_{J+j} } \bigg| , \\
h_1 ( \w)   = &\  h_0 ( a_1 , w_2, \dots, w_{2J}), \\
h_2 ( \w )   =&\   h_0  ( a_1, a_2 , w_3, \dots, w_{2J}) ,\\
  \vdots  \  &  \\
  h_{2J} (\w )   = &\   h_0 ( a_1 , a_2 , \dots, a_{2J} ) .
\end{split}\end{equation}
Notice that $h_{2J} $ is a constant  function  and
$$   h_0 ( \w) = \sum_{j=0}^{2J-1} \big( h_j ( \w) - h_{j+1} (\w) \big)    + h_{2J} ( \w) .  $$
Thus we have
\begin{equation}\notag
\begin{split}
 \int_R   \log \bigg| 1 +   \sum_{ j <J} e^{u_j + i v_j } \bigg|         d\Psi_{T} ( \u, \v)   
= &\int_R  h_0 ( \w)       d\Psi_{T} ( \w)    \\
= & \int_R \Bigg(  \sum_{j=0}^{2J-1} \big( h_j ( \w) - h_{j+1} (\w) \big)    + h_{2J} ( \w)   \Bigg)  d\Psi_{T} ( \w) \\
= & \sum_{j=0}^{2J-1}  \int_R  \big( h_j ( \w) - h_{j+1} (\w) \big)     d\Psi_{T} ( \w) +   h_{2J} (\a )    \Psi_{T} (R) ,
\end{split}\end{equation}
where $ \a = ( a_1, \dots , a_{2J} )$. 
Now, letting $ \widetilde{\w}_{j+1} = ( w_1 , \dots , w_j , \widetilde{w}_{j+1} , w_{j+2} , \dots, w_{ 2J} ) $, we have
$$ h_j ( \w ) - h_{j+1} ( \w ) =  \int_{ a_{j+1}}^{w_{j+1}}   \frac{ \partial h_j ( \widetilde{\w}_{j+1})}{ \partial \widetilde{w}_{j+1} } d \widetilde{w}_{j+1} .$$
By  \eqref{lower bd 3}, \eqref{uj bound}, and \eqref{h formulas},
we have  
$$
 \frac{ \partial h_j ( \w)}{ \partial {w}_{j+1} }   
 \ll   \bigg|1+\sum_{1\leq j\leq J-1} e^{w_j+iw_{J+j}}\bigg|^{-1} \ll
 \delta^{-1}   
$$
on   $ R= \prod_{ j \leq 2J} ( a_j, b_j ] $.
Thus, if
$$ R_{j+1 } ( \widetilde{w}_{j+1})  := (a_1 ,b_1] \times \cdots \times (a_j, b_j] \times (\widetilde{w}_{j+1} , b_{j+1} ] \times   (a_{j+2} ,b_{j+2} ] \times \cdots \times (a_{2J-2}, b_{2J-2} ] , $$
we see that
\begin{equation}\label{eqn:2}
\begin{split}
   & \int_R  \big( h_j ( \w) - h_{j+1} (\w) \big)     d\Psi_{T} ( \w)  \\
   & =  \int_R  \bigg(\int_{ a_{j+1}}^{w_{j+1}}   \frac{ \partial h_j ( \widetilde{\w}_{j+1})}{ \partial \widetilde{w}_{j+1} } d \widetilde{w}_{j+1}\bigg) d\Psi_{T} ( \w)  \\
   & =   \int_{ a_{j+1}}^{b_{j+1}} \bigg( \int_{R_{j+1 } ( \widetilde{w}_{j+1})}   d\Psi_{T} ( \w)    
   \bigg)\frac{ \partial h_j ( \widetilde{\w}_{j+1})}{ \partial \widetilde{w}_{j+1} } d \widetilde{w}_{j+1}  \\
      & =   \int_{ a_{j+1}}^{b_{j+1}}  \Psi_{T} ( R_{j+1 } ( \widetilde{w}_{j+1})   )    \frac{ \partial h_j ( \widetilde{\w}_{j+1})}{ \partial \widetilde{w}_{j+1} } d \widetilde{w}_{j+1} .
      \end{split}
      \end{equation}
 Now  by Theorem~\ref{lem: discrepancy}, 
$$   \Psi_{T}(R_{j+1 } ( \widetilde{w}_{j+1})  ) - \Psi (R_{j+1 } ( \widetilde{w}_{j+1})  )  
=  O( ( \log T)^{-\sigma} ) . $$ 
Reversing our steps, we see that the last line in \eqref{eqn:2} equals
\begin{equation}\notag
\begin{split}
      & =   \int_{ a_{j+1}}^{b_{j+1}}  \Psi  ( R_{j+1 } ( \tilde{w}_{j+1})   )    
      \frac{ \partial h_j ( \tilde{\w}_{j+1})}{ \partial \tilde{w}_{j+1} } d \tilde{w}_{j+1} + O\left(  \frac{ | b_{j+1}-a_{j+1} |  }{\delta (\log T)^{  \sigma } }      \right) \\
         & =   \int_{ a_{j+1}}^{b_{j+1}} \bigg( \int_{R_{j+1 } ( \tilde{w}_{j+1})}   d\Psi ( \w)   \bigg) \frac{ \partial h_j ( \tilde{\w}_{j+1})}{ \partial \tilde{w}_{j+1} } d \tilde{w}_{j+1} + O\left(   \frac{ \sL   }{\delta (\log T)^{  \sigma } }     \right) \\
            & =  \int_R  \bigg(\int_{ a_{j+1}}^{w_{j+1}}   \frac{ \partial h_j ( \tilde{\w}_{j+1})}{ \partial \tilde{w}_{j+1} } d \tilde{w}_{j+1} \bigg) d\Psi ( \w) + O\left(   \frac{ \sL   }{\delta (\log T)^{  \sigma } }       \right)  \\
       &=   \int_R  \big( h_j ( \w) - h_{j+1} (\w) \big)     d\Psi  ( \w) + O\left(   \frac{ \sL   }{\delta (\log T)^{  \sigma } }       \right) .
\end{split}\end{equation}
Hence,   we obtain
$$ \int_R   \log \bigg| 1 + \sum_{ j <J} e^{u_j + i v_j } \bigg|         d\Psi_{ T} ( \u, \v) = \int_R   h_0 ( \u,\v)     d\Psi  ( \u, \v) +  O\left(   \frac{ \sL  }{\delta (\log T)^{  \sigma } }       \right) . $$ 
Since $R_{main}(\m,\n) $ is a union of at most $O(\sL)$  rectangles $R$, we have
\begin{equation}\label{eqn:decompose main}
  \int_{ R_{main}(\m,\n) }   \log \bigg| 1 + \sum_{ j \leq J -1} e^{u_j + i v_j } \bigg|         d\Psi_{T} ( \u, \v) =  \int_{ R_{main}(\m,\n) }  h_0 ( \u,\v)        d\Psi ( \u, \v)  +  O\left(  \frac{ \sL^2  }{\delta (\log T)^{  \sigma } }    \right) .
  \end{equation}

\noindent{\bf The error term:} \
 We bound
 $$ \frac1T \int_{S_{error} (T)} \log \bigg|1 + \sum_{ j < J }  \frac{c_j}{c_J}  \frac{L_j }{ L_J} (\sigma+it) \bigg| dt  
 $$
 by using  the Cauchy-Schwarz inequality and  showing that ${\rm meas}(S_{ error}(T))$ is small.  

If $ t \in S_{error}(T)$, then $\L( \sigma+it) \in R_{error}(\m,\n)$ for some $\m,\n$. Thus, by 
\eqref{upper bd 3}
$$ \log  \bigg|1 + \sum_{   j < J }  \frac{c_j}{c_J}  \frac{L_j }{ L_J} (\sigma+it) \bigg| \leq  \log (23J\delta)= -c_0 \sL + \log (23J). $$
By Lemma \ref{lem: 2kth mmt real} 
\begin{equation*}\begin{split}
& \frac1T \int_T^{2T} \bigg|  \log \bigg|1 + \sum_{   j < J }  \frac{c_j}{c_J}  \frac{L_j }{ L_J} (\sigma+it) \bigg| \bigg|^{2k} dt \\
& \leq  4^k \frac1T \int_T^{2T} \bigg|  \log \bigg|  \sum_{   j \leq J }   c_j   L_j   (\sigma+it) \bigg| \bigg|^{2k} dt + 4^k  \frac1T \int_T^{2T} \bigg|  \log \bigg| c_J  L_J  (\sigma+it) \bigg| \bigg|^{2k} dt \\
& \ll (Ck)^{4k}.
\end{split}\end{equation*}
Hence, 
 \begin{align*}
{\rm meas} (S_{error}(T))\leq& \int_{S_{ error}(T)} \Bigg(\frac{ \log  \big|1 + \sum_{  j < J }  \frac{c_j}{c_J}  \frac{L_j }{ L_J} (\sigma+it) \big|}{c_0 \sL}\Bigg)^{2k} dt \\
\ll &  T \bigg(\frac{C^2 k^2}{c_0\sL}\bigg)^{2k} .
 \end{align*} 
Taking $ k = a \sqrt{ \sL} $ for  any $0< a<\sqrt{c_0}/C $, we see that  
 \begin{equation}\label{eqn:decompose error}
\frac1T | S_{ error}(T) | \ll   e^{- 2b \sqrt{\sL}}
\end{equation}
with $ b =  2a  \log ( \sqrt{c_0}/aC ) >0 $.  
It now follows  by the Cauchy-Schwarz inequality and Lemma~\ref{lem: 2kth mmt real} 
(with $k=1$) that
\begin{align*}
\frac{1}{T}  \Bigg|\int_{S_{ error}(T)}& \log  \bigg| 1 + \sum_{  j < J }  \frac{c_j}{c_J} 
   \frac{{L}_j }{ L_J} (\sigma+it)   \bigg| dt  \Bigg| \\
   \leq&\bigg(\frac1T {\rm meas(S_{  error}(T))}\bigg)^{\frac12}
  \Bigg( \frac{1}{T}  \int_{T}^{2T} \bigg( \log  \bigg| 1 + \sum_{   j < J }  \frac{c_j}{c_J} 
   \frac{{L}_j }{ L_J} (\sigma+it)   \bigg|\bigg)^2 dt   \Bigg)^{\frac12}\\
   \ll& e^{-b\sqrt{\sL}}.
\end{align*}

By \eqref{eqn:decompose}, \eqref{eqn:decompose main}, and \eqref{eqn:decompose error} we now see that
\begin{equation}\label{eqn:me1} \begin{split}
   \frac1T \int_{S_{J,2} (T)}  \log& \bigg|1 + \sum_{   j < J }  \frac{c_j}{c_J} 
   \frac{{L}_j }{ L_J } (\sigma+it) \bigg| dt  \\
 = & \sum_{ \m,\n}  \int_{ R_{main}(\m,\n) }  h_0 ( \u,\v)        d\Psi  ( \u, \v) 
  +  O\left(   \frac{\sL^{J+1}  }{  (\log T)^{  \sigma - (2J-3)c_0 } }    \right)   + O(  e^{- b \sqrt{\sL}})  \\
   = & \sum_{ \m,\n}  \int_{ R_{main}(\m,\n) }  h_0 ( \u,\v)        d\Psi ( \u, \v) 
    + O(  e^{- b \sqrt{\sL}})  .
\end{split}\end{equation}

Next we   show that
\begin{equation}\label{error integral}
    \int_{ R_{error}  }  h_0 ( \u,\v)        d\Psi ( \u, \v) =O(  e^{- b \sqrt{\sL}})  ,
\end{equation}
where
\begin{equation}\label{eqn:error}
 R_{error} = \bigcup_{\m,\n} R_{error}(\m,\n). 
 \end{equation}
It will then follow from \eqref{set decomp}  and \eqref{eqn:me1} -- \eqref{eqn:error} that 
 \begin{equation}\label{S12T}\begin{split}
\frac1T  \int_{S_{J,2} (T)} \log \bigg|1 + \sum_{   j < J }  \frac{c_j}{c_J} 
 \frac{L_j }{ L_J} (\s+it) \bigg| dt  
 = & \int_{ (I^{-}_T)^{J-1} \times I_T  \times (I_T)^J }   h_0 (\u, \v)     d\Psi ( \u, \v) \\
 & + O(  e^{- b \sqrt{\sL}}).
\end{split}\end{equation}
 
 To prove \eqref{error integral}, we first note  that by H\"older's inequality
\begin{align*}
  \left|   \int_{ R_{error}  }  h_0 ( \u,\v)        d\Psi ( \u, \v) \right|  
   \leq \left|  \int_{ R_{error}  }  |h_0 ( \u,\v)|^{2k}        d\Psi ( \u, \v) \right|^{1/2k} \left|  \int_{ R_{error} }         d\Psi ( \u, \v) \right|^{1-1/2k}. 
 \end{align*}
By  Lemma~\ref{2kth mmt random}, the first integral on the right is
\begin{align}\label{h_0 Rerror 2k}
  \left|   \int_{ R_{error}  }  |h_0 ( \u,\v)|^{2k}   d\Psi  ( \u, \v) \right| 
  & \leq  \E \Bigg[  \bigg|  \log \bigg|1 + \sum_{  j < J }  \frac{c_j}{c_J}  \frac{L_j }{ L_J} (\sigma,X) \bigg|\, \bigg|^{2k}    \Bigg] \\
   & \leq 4^k \Bigg( \E \Big[  \big|  \log \big|c_J L_J (\sigma,X) \big| \big|^{2k}    \Big] 
   + \E \Bigg[  \bigg|  \log \bigg| \sum_{  j \leq J }   c_j  L_j  (\sigma,X) \bigg| \,\bigg|^{2k} 
      \Bigg] \Bigg)  \notag   \\
   & \ll ( Ck)^{2k}.\notag
 \end{align}
Hence
\begin{align}\label{h_0 Rerror}
  \left|   \int_{ R_{error}  }  h_0 ( \u,\v)        d\Psi  ( \u, \v) \right|  
   \ll  k  \left|  \int_{ R_{error} }         d\Psi  ( \u, \v) \right|^{1-1/2k}. 
 \end{align}
By \eqref{upper bd 3} we see that for $ (\u, \v) \in R_{error}$,  
 $$h_0 ( \u, \v ) 
 = \log \bigg| 1 +f \sum_{ j <J} e^{u_j + i v_{j} } \bigg|
 \leq \log (23J\delta)  \leq -c_0 \sL+ \log ( 23 J).$$
Thus
$$ | h_0 ( \u,\v) | \geq c_0 \sL - \log ( 23 J) ,$$
and we obtain 
\begin{align*}
0 \leq \Psi(R_{error}) =  \int_{ R_{error} }         d\Psi  ( \u, \v) 
 &  \leq  \int_{R_{error}}  \frac{  | h_0 ( \u,\v) |^{2k}}{ (c_0 \sL - \log ( 23 J ) )^{2k}}\, d\Psi  ( \u, \v)  \\
&  \ll \left(   \frac{ Ck}{ c_0 \sL - \log ( 23 J) } \right)^{2k},
\end{align*}
by \eqref{h_0 Rerror 2k}.
Using this in \eqref{h_0 Rerror}, we obtain
\begin{align*}
  \left|   \int_{ R_{error}  }  h_0 ( \u,\v)        d\Psi ( \u, \v) \right|  
   \ll  \sL \left(   \frac{ Ck}{ c_0 \sL - \log ( 23 J) } \right)^{2k} . 
 \end{align*}
We now take $k = (c_0 \sL - \log (23J) )/(eC) $ and find that
\begin{align*}
  \left|   \int_{ R_{error}  }  h_0 ( \u,\v)        d\Psi ( \u, \v) \right|  
   \ll  \sL e^{- 2(c_0  / eC) \sL } \ll e^{-b \sqrt{\sL} } . 
 \end{align*}
This proves \eqref{error integral} and thus \eqref{S12T}.

\subsection{Completion of the proof of Theorem~\ref{main thrm} }

By \eqref{eqn: E int split}, \eqref{S1 int 2}, and  \eqref{S12T}, 
\begin{equation}\notag
\begin{split}
 \frac 1T \int_{S_J (T)} \log | E( \sigma + it, Q) | dt 
   =& \int_{  (I^- _T )^{J-1} \times I_T \times (I_T)^J } 
 \bigg( u_J + \log \bigg| 1+ \sum_{j<J} e^{u_j + iv_J } \bigg| \bigg) d \Psi (\u, \v) \\
 &+ O( e^{-b \sqrt{ \sL}} ) \\
   =& \E \bigg[   {\bold 1}_{\mathscr I_{J,T}}\cdot \bigg(  \log| c_J L_J ( \sigma ,X ) | + \log \bigg| 1 + \sum_{j < J } \frac{ c_jL_j}{c_J L_J} ( \sigma ,X ) \bigg|  \bigg) \bigg] \\
 &+ O( e^{-b \sqrt{ \sL}} ) \\ 
 & = \E \bigg[   {\bold 1}_{\mathscr I_{J,T}}\cdot  \log \bigg|   \sum_{j \leq J }   c_jL_j  ( \sigma ,X) \bigg| \bigg]+ O( e^{-b \sqrt{ \sL}} ),
 \end{split}\end{equation}
where $ \mathscr I _{J,T} $ is the event
$$   \L ( \sigma ,X ) \in   (I^- _T )^{J-1} \times I_T \times (I_T)^J , $$
and $   {\bold 1}_{\mathscr I_{J,T}} $ is its characteristic function. By Lemma \ref{2kth mmt random} 
$$ \E \bigg[   {\bold 1}_{\mathscr I_{J,T}}\cdot  \log \bigg|   \sum_{j \leq J }   c_jL_j  ( \sigma ,X) \bigg| \bigg]  + O( e^{-b \sqrt{ \sL}} )  = \E \bigg[   {\bold 1}_{\mathscr I_J }\cdot  \log \bigg|   \sum_{j \leq J }   c_jL_j  ( \sigma ,X) \bigg| \bigg]  + O( e^{-b \sqrt{ \sL}} ) , $$
where
 $ \mathscr I _{J} $ is the event
$$   \L ( \sigma ,X ) \in   (-\infty, 0]^{J-1} \times \mathbb R^{J+1} ,$$
and $   {\bold 1}_{\mathscr I_J } $ is its characteristic function. Hence $ \mathscr I_J $ is the event
$$ | L_J ( \sigma , X ) | = \max_{j }  | L_j ( \sigma , X)| . $$
Reversing all the steps in \eqref{log E est 1},    we therefore obtain
\begin{align}\label{log E int}
  \frac1T \int_T^{2T} \log |E (\sigma+it, Q)| dt  
 = &  \sum_{  j \leq J} \frac1T \int_{S_j(T)} \log |E(\sigma+it, Q)| dt + O \bigg( \frac{ \sL^2}{ ( \log T)^\s } \bigg) \notag \\
=  &  \sum_{  j \leq J}   \E \bigg[   {\bold 1}_{\mathscr I_j }\cdot  \log \bigg|   \sum_{j \leq J }   c_jL_j  ( \sigma ,X) \bigg| \bigg]  + O( e^{-b \sqrt{ \sL}} )  \\
= & \mathcal{M}(\sigma) + O( e^{- b \sqrt{ \sL }}),\notag
\end{align}
where
$$ \mathcal{M} (\sigma) = \E \bigg[   \log \bigg| \sum_{   j \leq J }   c_j  L_j (\sigma , X ) \bigg|   \bigg] .$$
 In~\cite{Le2} the second author proved that $\mathcal{M} (\sigma)$ is twice differentiable. We use this   to show that 
 \begin{equation}\label{eqn: zero density}
  N_E ( \sigma_1, \sigma_2 ; T) = - \frac{T}{2 \pi} ( \mathcal{M}'(\sigma_1) - \mathcal{M}' ( \sigma_2 ) ) +   O( e^{-  ({ b}/{2}) \sqrt{ \sL}}). 
\end{equation}

Applying  Littlewood's lemma in a standard way, we find that
$$ 
 \int_{\sigma}^{\sigma_0} 
\bigg(\sum_{ \substack{ \beta > u \\ T \leq \gamma \leq 2T }}  1 \bigg) du 
= \frac{1}{2 \pi} \int_T^{2T} \log |   E( \sigma+it, Q) | dt
-\frac{1}{2 \pi} \int_T^{2T} \log |   E( \sigma_0+it, Q) | dt + O( \log T)  ,
$$
where $ \sigma_0 $ is a large but fixed real number such that $E(s,Q)$ has no zero in $ \Re s \geq \sigma_0 $. By \eqref{log E int} 
$$  \int_{\sigma}^{\sigma_0} 
\bigg(\sum_{ \substack{ \beta > u \\ T \leq \gamma \leq 2T }}  1 \bigg)  du 
= \frac{T}{2 \pi}\mathcal{M} (\sigma) 
- \frac{1}{2 \pi} \int_T^{2T} \log |   E( \sigma_0+it, Q) | dt
  + O(T e^{- b \sqrt{ \sL }}). $$
Differencing this at $\s$ and $\s+h$,  with $h>0$ small, we deduce that
\begin{align}  
 \frac1h \int_{\sigma}^{\sigma+h} 
\bigg(\sum_{ \substack{ \beta > u \\ T \leq \gamma \leq 2T }}  1 \bigg)  du 
=\frac{T}{2 \pi}\frac{\mathcal{M} (\sigma) -\mathcal{M} (\sigma+h) }{h}
  + O\Big(\frac{T}{h} e^{- b \sqrt{ \sL }}\Big). 
\end{align}   
Since $\mathcal{M} (\sigma)$ is twice differentiable, we may write this as 
$$
\frac1h \int_{\sigma}^{\sigma+h} 
\bigg(\sum_{ \substack{ \beta > u \\ T \leq \gamma \leq 2T }}  1 \bigg)  du =
-\frac{T}{2 \pi} \mathcal{M}' (\sigma)  
  + O\Big(hT+\frac{T}{h} e^{- b \sqrt{ \sL }}\Big). 
$$
The integrand is a non increasing function of $u$, so 
 \begin{align*}
 \sum_{ \substack{ \beta >\sigma+h  \\ T \leq \gamma \leq 2T }}  1   
 & \leq  - \frac{T}{2 \pi}  \mathcal{M}'(\sigma ) 
 +  O\Big(hT+\frac{T}{h} e^{- b \sqrt{ \sL }}\Big)
 \leq  \sum_{ \substack{ \beta >\sigma  \\ T \leq \gamma \leq 2T }}  1  . 
   \end{align*}
On the left-hand side  we replace $\s$ by $\s-h$ and use 
$\mathcal{M}'(\sigma-h ) =\mathcal{M}'(\sigma ) +O(h)$. We then find that
 \begin{align*}
 \sum_{ \substack{ \beta >\sigma   \\ T \leq \gamma \leq 2T }}  1   
 & =  - \frac{T}{2 \pi}  \mathcal{M}'(\sigma ) 
 +  O\Big(hT+\frac{T}{h} e^{- b \sqrt{ \sL }}\Big) . 
   \end{align*}
Taking $h= e^{- \frac{ b}{2} \sqrt{ \sL}}$, we obtain
 $$  \sum_{ \substack{ \beta> \sigma  \\ T \leq \gamma \leq 2T }}  1 
 =  - \frac{T}{2 \pi}   \mathcal{M}'(\sigma)     + O(Te^{- (b/2)  \sqrt{ \sL} }).
 $$
 Equation  \eqref{eqn: zero density} follows easily from this.


\noindent Department of Mathematics, University of Rochester, Rochester, NY 14627, USA\\
\textsl{E-mail address} : Steven Gonek (gonek@math.rochester.edu)

\noindent Department of Mathematics, Incheon National University, \\
119 Academy-ro, Yeonsu-gu, Incheon, 22012, Korea\\
\textsl{E-mail address} : Yoonbok Lee (leeyb@inu.ac.kr, leeyb131@gmail.com)

\end{document}